\documentclass[200pt]{amsart}
\usepackage{url, amsmath, amsthm, amssymb,mathtools ,amsmath,ifthen, url, amsmath, amsthm, amssymb,url, amssymb, relsize,placeins, tikz, pgffor, multicol,mathrsfs, abstract, graphicx, fp, hyperref, rotating, blindtext}

\makeatletter
\renewcommand\part{%
   \if@noskipsec \leavevmode \fi
   \par
   \addvspace{4ex}%
   \@afterindentfalse
   \secdef\@part\@spart}

\def\@part[#1]#2{%
    \ifnum \c@secnumdepth >\m@ne
      \refstepcounter{part}%
       \ifnum \thepart < 4
            \addcontentsline{toc}{part}{\thepart\hspace{1em}#1}%
       \else
             \addcontentsline{toc}{part}{A\hspace{1em}#1}%
       \fi
    \else
      \addcontentsline{toc}{part}{#1}%
    \fi
    {\parindent \z@ \raggedright
     \interlinepenalty \@M
     \normalfont
     \ifnum \c@secnumdepth >\m@ne
       \Large\bfseries \nobreakspace
       \par\nobreak
     \fi
     \huge \bfseries #2%
     \par}%
    \nobreak
    \vskip 3ex
    \@afterheading}
\def\@spart#1{%
    {\parindent \z@ \raggedright
     \interlinepenalty \@M
     \normalfont
     \huge \bfseries #1\par}%
     \nobreak
     \vskip 3ex
     \@afterheading}
\makeatother
\newtheorem{theorem}{Theorem}[section]
\newtheorem{remark}[theorem]{Remark}
\newtheorem{lemma}[theorem]{Lemma}
\newtheorem{question}[theorem]{Question}
\newtheorem{proposition}[theorem]{Proposition}
\newtheorem{corollary}[theorem]{Corollary}

\newenvironment{defn}[1][]{\refstepcounter{theorem}\begin{trivlist}
\item[\hskip \labelsep {\bfseries Definition  \thetheorem  \, \def\temp{#1}\ifx\temp\empty  #1\else  (#1)\fi
}]}   {\end{trivlist}}
\newenvironment{example}[1][]{\refstepcounter{theorem}\begin{trivlist}
\item[\hskip \labelsep {\bfseries Example  \thetheorem  \, \def\temp{#1}\ifx\temp\empty  #1\else  (#1)\fi
}]}   {\end{trivlist}}
\usepackage{multicol}
\usepackage{hyperref}
\usepackage{multicol}
\usepackage{mathrsfs}
\usepackage{abstract}
\usepackage{graphicx}
\usepackage{tikz}
\usetikzlibrary{positioning,shapes,shadows}
\usetikzlibrary{arrows,automata}
\hypersetup{
    colorlinks,
    citecolor=black,
    filecolor=black,
    linkcolor=black,
    urlcolor=black
}

\newcommand{\Aut}{\operatorname{Aut}}
\newcommand{\Out}{\operatorname{Out}}
\newcommand{\Z}{\mathbb{Z}}
\newcommand{\N}{\mathbb{N}}
\newcommand{\im}{\operatorname{img}}
\newcommand{\fim}{\operatorname{fimg}}
\newcommand{\bim}{\operatorname{bimg}}
\newcommand{\adj}{\operatorname{Adj}}
\newcommand{\rel}{\operatorname{Rel}}
\newcommand{\dom}{dom}

\newcommand{\norm}[1]{\left\lVert#1\right\rVert}
\usepackage[whole]{bxcjkjatype}
\usepackage{newunicodechar}

\newcommand{\PF}{\operatorname{PF}}
\newcommand{\Hom}{\operatorname{Hom}}
\newcommand{\Wlk}{\operatorname{Walks}}
\newcommand{\GBwlk}{\operatorname{NBWalks}}
\newcommand{\rt}{\operatorname{Route}}
\newcommand{\Flt}{\operatorname{Flat}}
\newcommand{\makeset}[2]{\left\lbrace #1 \;\middle|\;
  \begin{tabular}{@{}l@{}}
    #2
   \end{tabular}
  \right\rbrace}
  \title[Unindexed subshifts of finite type]{Unindexed subshifts of finite type and their connection to automorphisms of Thompson's groups}
\date{March 2022}

\thanks{The author is grateful for the partial support of the EPSRC research grant EP/R032866/1}
\begin{document}
\vspace*{-40pt}
\maketitle
\begin{center}
\author{Luke Elliott}
\end{center}
\tableofcontents

\begin{abstract}
For a finite digraph \(D\), we define the corresponding two-sided subshift of finite type \((X_D, \sigma_D)\) to be the dynamical system where \(X_D\) is the set of all bi-infinite walks through \(D\) and \(\sigma_D\) is the shift operator.
Two digraphs \(D_1\) and \(D_2\) are called \textit{shift equivalent} if there is \(m\geq 1\) such that \((X_{D_1}, \sigma_{D_1}^n)\) and \((X_{D_2}, \sigma_{D_2}^n)\) are topologically conjugate for all \(n\geq m\). They are called \textit{strong shift equivalent} if this holds for \(m=1\).

In this paper we introduce a new category which generalises the category of subshifts of finite type and topological conjugacy.
Our category gives two new notions of equivalence for digraphs which we call strong UDAF equivalence and weak UDAF equivalence. 
Strong UDAF equivalence is a coarser analogue of strong shift equivalence and weak UDAF equivalence is a coarser analogue of shift equivalence.
Both strong and weak UDAF equivalence still separate the \(n\)-leaf roses for all \(n\geq 2\) (the 1-vertex \(n\)-edge digraphs).
However, strong UDAF equivalence does not imply shift equivalence, in particular it equates the 2-leaf rose with the golden mean shift.

We also show that our category relates to the automorphism groups of the ``\(V\)-type" groups of Higman, Thompson and Brin. 
In particular for all \(n\geq 2\) the groups \(\Out(G_{n, n-1})\) and \(\Out(nV)\) can be seen as automorphism groups of specific objects in our category.

We explain a few equivalent ways of viewing strong UDAF equivalence, and give an example of how they can be used to show that Ashley's 8 vertex digraph is strong UDAF equivalent to the 2-leaf rose. 
\end{abstract}
\newpage
\section{Introduction}

In this paper we introduce two new categories UDAF and UDAFG. 
Here UDAF stands for ``unindexed digraphs and foldings" and the G stands for ``groupoid" (as in a category with inverses).
The UDAFG category serves as a generalisation of the category of two-sided subshifts of finite type and topological conjugacy (modulo the restriction to the digraphs in Definition~\ref{udaf digraph defn}) which at the same time gives relatively straightforward descriptions of the outer automorphism groups of \(nV\) and \(G_{n, n-1}\). 
In particular, further investigation of UDAFG could yield progress in both of these areas.
This link continues the existing narrative of James Belk, Collin Bleak, Peter J. Cameron, and Feyishayo Olukoya in \cite{belk2020automorphisms, bleak2020automorphisms}. 
In particular, the authors show that the automorphism group of the one-sided shift on \(n\) letters naturally embeds in \(\Out(G_{n, n-1})\), and similarly the automorphism group of the two-sided shift on \(n\) letters is an extension of a subgroup of \(\Out(G_{n, n-1})\) by \(\Z\).

Instead of considering homeomorphisms between subshifts of finite type \((X, \sigma_X)\) and \((Y, \sigma_Y)\) which preserve the shift operator, UDAFG considers particularly nice bijections between the sets of orbits of these systems. 
This gives us a courser notion of equivalence than the usual strong shift equivalence between subshifts. 
We call this notion of equivalence \textit{strong UDAF equivalence}.
We also define a weaker notion of equivalence called \textit{weak UDAF equivalence} analogous to shift equivalence and show the following (non-)implications hold:

\vspace*{30pt}\begin{rotate}{00}\hspace*{-27pt}\(\text{Weak UDAF Equivalence}\)

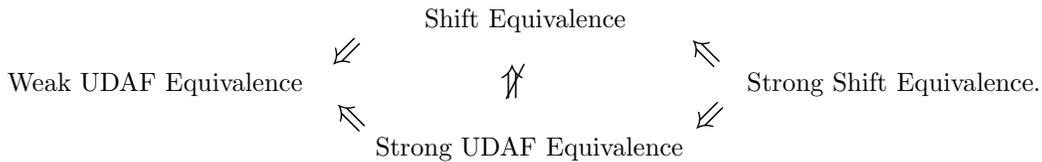
\begin{rotate}{-45}\(\quad\mathlarger{\mathlarger{\mathlarger{\Leftarrow}}}\quad
\)\begin{rotate}{45}\(\text{Strong UDAF Equivalence}
\)\begin{rotate}{45}\(\quad\mathlarger{\mathlarger{\mathlarger{\Leftarrow}}}\quad
\)\begin{rotate}{-45}\(\text{Strong Shift Equivalence}.
\)\end{rotate}\end{rotate}\end{rotate}\end{rotate}

 \begin{rotate}{45}\(\quad\mathlarger{\mathlarger{\mathlarger{\Leftarrow}}}\quad
\)\begin{rotate}{-45}\hspace*{12pt} \(\text{Shift Equivalence}
\)\hspace*{12pt} \begin{rotate}{-45}\(\quad\mathlarger{\mathlarger{\mathlarger{\Leftarrow}}}\quad
\)\end{rotate}\end{rotate}\end{rotate}
\hspace*{65pt}\(\mathlarger{\mathlarger{\mathlarger{\not\Uparrow}}}\)
\end{rotate}
\vspace*{30pt}

It remains open whether or not shift equivalence implies strong UDAF equivalence. 
If not, then strong UDAF equivalence may serve as a useful tool for separating shift and strong shift equivalence.

In the following two theorems, we see that the outer automorphism groups of Thompson's group \(V\) and related groups can be seen as the UDAFG automorphism groups of simple digraphs.

The Higman-Thompson groups \(G_{n, r}\) (\(n\geq 2\) and \(r\geq 1\)) are a family of infinite finitely presented groups. Each of these groups is either simple or has a simple subgroup of index 2, and moreover these were the first infinite family of such groups. 
For more background on \(G_{n, r}\) see \cite{pardo2010isomorphism, autgnr, thumann2016operad, higman1974finitely}.

\begin{theorem}
If \(n \geq 2\), then we denote by \(R_n\) the digraph with \(1\) vertex and \(n\) edges.
For each such \(n\), the automorphism group of the UDAFG object corresponding to \(R_n\) is isomorphic to the outer automorphism group of \(G_{n, n-1}\).
\end{theorem}

The Brin-Thompson groups \(nV\) (\(n\geq 1\)) serve as \(n\)-dimensional analogous to Thompson's group \(V\), in particular they are all distinct infinite simple and finitely presented (see \cite{Brin2004, Brin2005, Bleak2010, quick2019, autnv} of more details).
\begin{theorem}
If \(n\in \Z\) and \(n \geq 1\), then we denote by \(nR_2\) the digraph consisting of \(n\) disjoint copies of \(R_2\).
For each such \(n\), the automorphism group of the UDAFG object corresponding to \(nR_2\) is isomorphic to the outer automorphism group of \(nV\).
\end{theorem}

In the ``beams and multisets" part of the document, we develop a theory of unindexed rays and beams analogous to the existing theory of rays and beams for subshifts of finite type (see section 7.5 of \cite{lind2021introduction}). 
We use this to associate to a finite digraph a commutative semigroup which we call the \textit{UDAF dimension semigroup} analogous the the dimension modules for shift equivalence (by Theorem~\ref{UDAF dimension group}, this semigroup is often a group).
We show that a presentation for the UDAF dimension semigroup of a finite digraph can be easily deduced from the digraph directly and that the isomorphism type can be inferred from the digraph's dimension module.
In particular we show that weak UDAF equivalence is implied by shift equivalence.

The UDAF category is a generalisation of the category of digraphs and foldings introduced by Jim Belk, Collin Bleak, and Peter Cameron.
It was originally made as a generalisation of their category which allows for Bleak's notion of ``conglomeration". 
As of the time this document was written, their category does not currently appear in the literature, so we will give the relevant definitions here.
Foldings are of particular note due to their connection with the problem of checking strong shift equivalence of subshifts of finite type.
In particular it is true that two finite digraphs are strong shift equivalent if and only if there is a third finite digraph which folds onto each them. 
We will show that the analogous statement is true for UDAF foldings and strong UDAF equivalence.

We also show that two matrices represent strong UDAF equivalent digraphs if and only if we can convert one into the other via a sequence of operations representable by matrix manipulations.

\begin{theorem}[Strong UDAF equivalence of matrices]\label{matequivth1}
Suppose that \(D_1\) and \(D_2\) are UDAF digraphs. 
Let \(A\) and \(B\) be matrices obtained from adjacency matrices of \(D_1\) and \(D_2\) respectively by subtracting \(1\) from the diagonal entries.

The digraphs \(D_1\) and \(D_2\) are strong UDAF equivalent if and only if there is a sequence of matrices
\[A=M_0, M_1, \ldots, M_{n}= B,\]
with entries from \(\N\) (or possibly \(-1\) for diagonal entries) such that each \(M_i\) can be obtained from \(M_{i-1}\) by doing one of the following:
\begin{enumerate}
    \item Adding a cross centered on the diagonal with \(-1\) on the diagonal and zeros elsewhere.
    \item Adding row \(j\) to row \(i\) when \(i\neq j\) and the \((i, j)\) entry is non-zero.
    \item Adding column \(i\) to column \(j\) when \(i\neq j\) and the \((i, j)\) entry is non-zero. 
    \item The inverse of any of the above operations.
    \item Reordering/relabeling the vertices of the digraph. In other words, suppose that \(M_i:V_i^2\to \Z\) and \(M_{i+1}:V_{i+1}^2\to \Z\) are matrices representing digraphs with vertex sets \(V_i\) and \(V_{i+1}\) respectively.
    Find a bijection \(b:V_i\to V_{i+1}\) such that for all \(v,w\in V_i\) we have \((v, w)M_i= ((v)b, (w)b)M_{i+1}\). One could also think of this as conjugating by a permutation matrix.
\end{enumerate}
\end{theorem}
The above theorem can be used to show that the golden mean shift is strong UDAF equivalent to the 2-leaf rose using only 3 steps (\((2)\), \((3)\), and \((1)\) inverse). 
We conclude by giving a more complex example of how one can use Theorem~\ref{matequivth1} to show that two digraphs are strong UDAF equivalent.
In particular we consider the 2-leaf rose and Ashley's \(8\times 8\) example (see Example 2.27 of \cite{kitchens1997symbolic}). 
Ashley's digraph is known to be shift equivalent to the two leaf rose \(R_2\) but it is not known whether or not they are strong shift equivalent.
\begin{center}
    \textbf{Acknowledgments}
    
    I would like to thank Collin Bleak for all his helpful feedback during the writing of this document.
\end{center}
\section{Definitions}
In this section we introduce terminology that we will be using throughout the paper.
In particular we explain how we will consider digraphs, and how we consider unindexed bi-infinite walks.
\begin{defn}
A \textit{digraph} \(D\) is a 4-tuple \((V_D, E_D, s_D, t_D)\) where \(s_D,t_D:E_D \to V_D\) are functions.
We call \(V_D\) and \(E_D\) the \textit{vertex} and \textit{edge} sets of \(D\) respectively, and we say that \(e\in E_D\) is an edge \textit{from} \((e)s_D\) \textit{to} \((e)t_D\).
\end{defn}
\begin{defn}
A \textit{digraph homomorphism} \(\phi:D_1 \to D_2\) between digraphs is a pair of functions \((\phi_V, \phi_E)\), where \(\phi_V:V_{D_1}\to V_{D_2}\) and \(\phi_E:E_{D_1} \to E_{D_2}\) satisfy
\(\phi_Es_{D_2} = s_{D_1}\phi_V\) and \(\phi_Et_{D_2}=t_{D_1}\phi_V\). 
That is to say that the following diagrams commute.

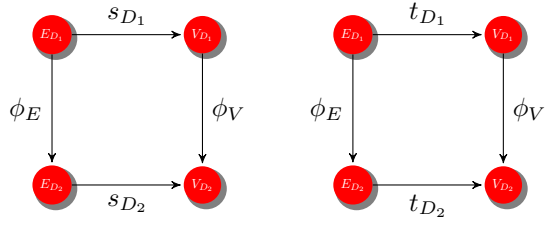
\begin{center}
\begin{tikzpicture}[->,>=stealth',shorten >=1pt,auto,node distance=5cm,on grid, every state/.style={fill=red,draw=none,circular drop shadow,text=white}]

    \node [state, scale = 0.5] (A) at (-3,1) {\(E_{D_1}\)};
  \node [state, scale = 0.5] (B) at (-1,1) {\(V_{D_1}\)};
  \node [state, scale = 0.5] (C) at (-3, -1) {\(E_{D_2}\)};
  \node [state, scale = 0.5] (D) at (-1, -1) {\(V_{D_2}\)};
 
  \node [state, scale = 0.5] (A2) at (1,1) {\(E_{D_1}\)} ;
  \node [state, scale = 0.5] (B2) at (3,1) {\(V_{D_1}\)};
  \node [state, scale = 0.5] (C2) at (1, -1) {\(E_{D_2}\)}; 
  \node [state, scale = 0.5] (D2) at (3, -1) {\(V_{D_2}\)};

 \path [->]
 (A) edge  [scale = 2]
 node  [swap]{\(\phi_{E}\)} (C)
 (A) edge  [out=0,in=180, scale = 2]
 node {\(s_{D_1}\)} (B)
 (B) edge  [ scale = 2]
 node {\(\phi_{V}\)} (D)
 (C) edge  [scale = 2]
 node [swap]{\(s_{D_2}\)} (D)

 (A2) edge  [scale = 2]
 node  [swap]{\(\phi_{E}\)} (C2)
 (A2) edge  [out=0,in=180, scale = 2]
 node {\(t_{D_1}\)} (B2)
 (B2) edge  [ scale = 2]
 node {\(\phi_{V}\)} (D2)
 (C2) edge  [scale = 2]
 node [swap]{\(t_{D_2}\)} (D2);
 \end{tikzpicture}
 \end{center}

Digraph homomorphisms are composed coordinate-wise. If \(D_1\) and \(D_2\) are digraphs, then we define \(\Hom(D_1, D_2)\) to be the set of digraph homomorphisms from \(D_1\) to \(D_2\) and \(\Aut(D_1)\) to be the automorphism group of \(D_1\).
\end{defn}

\begin{defn}
For \(n, m\in \Z\cup\{-\infty, +\infty\} \) with \(n\leq m\), we define \(\Z_{[n, m]}\) to be the digraph with vertex set \(\{i\in \Z: n\leq i\leq m\}\) and edge set \(\{(i, i+1): n\leq i< i+1\leq m\}\) where each edge \((i, i+1)\) goes from \(i\) to \(i+1\).
\end{defn}
For \(n, m\in \Z\), we think of \(\Z_{[n, m]}\) as the path between \(n\) and \(m\) in the integers. 
Indeed \(\Z_{[n, m]}\) is the restriction of the digraph \(\Z_{[-\infty, \infty]}\) to the interval \([n, m]\).
Note that in the above definition we allow the digraphs \(\Z_{[n,n]}\) (for \(n\in \Z\)), these digraphs have one vertex and zero edges.
\begin{defn}\label{concatenation definition}
Let \(D\) be a digraph. 
If \(a, d\in \Z\cup \{-\infty, \infty\}\), \(b,c \in \Z\), \(\alpha\in \Hom(\Z_{[a,b]}, D)\), \(\beta\in \Hom(\Z_{[c, d]}, D)\) and \((b)\alpha_V= (c)\beta_V\), then we define their \textit{concatenation} \(\alpha\star \beta: \Z_{[a, b+d-c]} \to D\) to be the digraph homomorphism which restricts to \(\alpha\) in \(\Z_{[a, b]}\) and restricts to a translation of \(\beta\) on \(\Z_{[b, b+d-c]}\).
That is for \(v\in V_{\Z_{[a, b+d-c]}}\) and \((w,w+1)\in E_{\Z_{[a, b+d-c]}}\) we have
\[(v)(\alpha\star \beta)_V = \begin{cases*}((v)\alpha_V & if \(v\leq b\)\\
(v+(c-b))\beta_V & if \(v\geq b\)
\end{cases*},\]
\[(w, w+1)(\alpha\star \beta)_E = \begin{cases*}(w, w+1)\alpha_E & if \(w< b\)\\
(w+(c-b), w+1+(c-b))\beta_E & if \(w\geq b\)
\end{cases*}.\]
\end{defn}
One of the ways our categories allow us to expand our notions of equivalence, is that they make it easy to forget the distinction between an edge and a finite walk.
Indeed we will be thinking of finite walks as edges for most of this document.

\begin{defn}\label{finite paths}
If \(D\) is a digraph, then we define \(\Wlk(D)\) to be the small category with
\begin{enumerate}
    \item  \(V_D\) as its set of objects.
    \item \(\bigcup_{ 0\leq n< \infty} \Hom(\Z_{[0,n]}, D)\) as its set of morphisms (where \(\alpha\in \Z_{[0, n]}\) is a morphism from \((0)\alpha_V\) to \((n)\alpha_V\)).
    \item Concatenation (\ref{concatenation definition}) as composition of morphisms.
\end{enumerate}
We will refer to the morphisms of this category as \textit{finite walks}.
If \(f\in \Hom(\Z_{[0,n]}, D)\), then we define the \textit{length} of \(f\) (denoted \(\norm{f}\)) to be  \(n\).
The set of identity morphisms of this category is then precisely the set \(\Hom(\Z_{[0,0]}, D)\) (the walks of length \(0\)).

 We will identify \(D\) with the natural embedded copy of \(D\) within \(\Wlk(D)\) with vertex set \(V_D\) and the walks of length \(1\) as edges.
 \end{defn}
 Immediately below and later, we will identify a small category with the underlying digraph obtained by using the objects as vertices and morphisms as edges (particularly in the case of the walks category in the previous definition).

\begin{remark}
If \(D\) is a digraph, then the category \(\Wlk(D)\) is free in the sense that if \(C\) is a small category and \(\phi:D\to C\) is a digraph homomorphism, then there is a unique functor \(\widetilde{\phi}:\Wlk(D) \to C\) extending \(\phi\).
These extensions will be used when we come to define UDAF foldings.
 \end{remark}

\begin{defn}
Similarly to the finite walks in Definition~\ref{finite paths}, if \(w:\Z_{[-\infty, \infty]}\to D\) is a digraph homomorphism, then we say that \(w\) a \textit{bi-infinite walk} (in D).
\end{defn}

In order to remove indexing we require a more general notion of a bi-infinite walk.
\begin{defn}
If \(w:\Z_{[-\infty, \infty]}\to \Wlk(D)\) is a digraph homomorphism then we say that \(w\) a \textit{generalised bi-infinite walk} (in D). 
Moreover we say that \(w\) is \textit{degenerate} if either
\begin{enumerate}
    \item There is \(z\in \Z\) such that for all \(n>z\) we have \(\norm{(n, n+1)w_E}= 0\).
    \item There is \(z\in \Z\) such that for all \(n<z\) we have \(\norm{(n, n+1)w_E}= 0\).
\end{enumerate}
We denote the set of non-degenerate bi-infinite walks in \(D\) by \(\GBwlk(D)\).
\end{defn}
Observe that the target object of a generalised bi-infinite walk, in a digraph \(D\), is \(\Wlk(D)\) and not \(D\) itself. We do this as being able to treat edges like walks in crucial throughout this document, and we often think of a bi-infinite walk as if it has been ``flattened" to form a bi-infinite walk in the more natural sense as explained in the following definition.
\begin{defn}
Let \(w\) be a non-degenerate generalised bi-infinite walk in a digraph \(D\).
We say that \(w_0:
\Z_{[-\infty, \infty]} \to D\) is a \textit{flattening} of \(w\) if there is a map \(f_{w, w_0}:\Z\to \Z\) such that for all \(z\in \Z\) and \(n\in \{0, 1, \ldots, \norm{(z, z+1)w_E}-1\}\) (where \(\{0, 1, \ldots, -1\}=\varnothing\)) we have:
\[(z)f_{w, w_0}+\norm{(z, z+1)w_E}= (z+1)f_{w, w_0},\]
\[((z)f_{w,w_0} + n)w_0=(n, n+ 1)((z, z+1)w_E)_E.\]

\begin{center}
    
 \begin{tikzpicture}[->,>=stealth',shorten >=1pt,auto,node distance=5cm,on grid,semithick, every state/.style={fill=red,draw=none,circular drop shadow,text=white}]
  \node [scale = 2] (W) at (-6,0) {\(w\)};
  \node [scale = 2] (Wrddd) at (-5,0) {\(\cdots\)};
  \node [state, scale = 0.7] (-1) at (-4, 0) {\(-1\)};
  \node [scale = 0.7] (-1d) at (-4, -0.5) {\((-1)w_V\)};
  \node [state, scale = 0.5] (-1u) at (-3.5, 1) {};
  
 \node [scale = 1.5] (-10) at (-2.5, 1) {\(\cdots\)};
 \node [scale = 1] () at (-2.5, 0.5) {\((-1, 0)w_E\)};
  \node [state, scale = 0.5] (0ul) at (-1.5, 1) {};
  \node [state, scale = 0.7] (0) at (-1, 0) {\(0\)};
   \node [scale = 0.7] (0d) at (-1, -0.5) {\((0)w_V\)};
  \node [state, scale = 0.5] (0ur) at (-0.5, 1) {};
 
 \node [scale = 1.5] (01) at (0.5, 1) {\(\cdots\)};
 \node [scale = 1] () at (0.5, 0.5) {\((0, 1)w_E\)};
 
  \node [state, scale = 0.5] (1ul) at (1.5, 1) {};
  \node [state, scale = 0.7] (1) at (2, 0) {\(1\)};
 \node [scale = 0.7] (1d) at (2, -0.5) {\((1)w_V\)};
  \node [state, scale = 0.5] (1ur) at (2.5, 1) {};
 \node [scale = 1.5] (12) at (3.5, 1) {\(\cdots\)};
 \node [scale = 1] () at (3.5, 0.5) {\((1, 2)w_E\)};
 
   \node [state, scale = 0.5] (2ul) at (4.5, 1) {};
  \node [state, scale = 0.7] (2) at (5, 0) {\(2\)};
 \node [scale = 0.7] (2d) at (5, -0.5) {\((2)w_V\)};
  \node [scale = 2] (2r) at (6, 0) {\(\cdots\)};
 
 \path [->]
 (-1) edge  [scale = 2]
 node  [swap]{} (-1u)
 (0ul) edge  [scale = 2]
 node  [swap]{} (0)
  (0) edge  [scale = 2]
 node  [swap]{} (0ur)
 (1ul) edge  [scale = 2]
 node  [swap]{} (1)
  (1) edge  [scale = 2]
 node  [swap]{} (1ur)
 (2ul) edge  [scale = 2]
 node  [swap]{} (2);
 \end{tikzpicture}
 
  \begin{tikzpicture}[->,>=stealth',shorten >=1pt,auto,node distance=5cm,on grid,semithick, every state/.style={fill=red,draw=none,circular drop shadow,text=white}]
  \node [scale = 2] (W) at (-6,0) {\(w_0\)};
  \node [scale = 2] (Wrddd) at (-5,0) {\(\cdots\)};
  \node [state, scale = 0.7] (-1) at (-4, 0) {\((0)f\)};
  \node [scale = 0.7] (-1d) at (-4, -0.5) {\((0)w_V\)};
  \node [state, scale = 0.7] (-1u) at (-2.8, 0) {};
 \node [scale = 1.5] (-10) at (-2, 0) {\(\cdots\)};
 \node [scale = 1] () at (-2, 0.5) {\((0, 1)w_E\)};
 \node [state, scale = 0.7] (0ul) at (-1.3, 0) {};
 \node [state, scale = 0.7] (0) at (-0.1, 0) {\((1)f\)};
 \node [scale = 0.7] (0d) at (-0.1, -0.5) {\((1)w_V\)};
 \node [state, scale = 0.7] (0ur) at (1.1, 0) {};
 
 \node [scale = 1.5] (01) at (1.9, 0) {\(\cdots\)};
 \node [scale = 1] () at (1.9, 0.5) {\((1, 2)w_E\)};
 
 \node [state, scale = 0.7] (1ul) at (2.7, 0) {};  
 \node [state, scale = 0.7] (1) at (3.9, 0) {\((2)f\)};
 \node [scale = 0.7] (1d) at (3.9, -0.5) {\((2)w_V\)};  \node [ scale = 2] (1ur) at (4.7, 0) {\(\cdots\)};
 
 \path [->]
 (-1) edge  [scale = 2]
 node  [swap]{} (-1u)
 (0ul) edge  [scale = 2]
 node  [swap]{} (0)
  (0) edge  [scale = 2]
 node  [swap]{} (0ur)
 (1ul) edge  [scale = 2]
 node  [swap]{} (1);
 \end{tikzpicture}
 
\end{center}
As no single flattening of a bi-infinite walk is most natural we will often talk about the set of all of them as opposed to a particular one. 
We then define \(\Flt(w)\) to be the set of flattenings of \(w\).
\end{defn}
\begin{example}
Let \(R_2\) denote the digraph \(V_D= \{\bullet\}\) and \(E_D= \{a, b\}\) (the 2-leaf rose). 
We will identify the finite walks in \(R_2\) with \(\{a, b\}^*\) in the intuitive fashion. 
Suppose that \(w:\Z_{[-\infty, \infty]}\) is the bi-infinite walk with \((z)w_V= \bullet\) for all \(z\), \((z, z+1)w_E= a\) whenever \(|z|> 1\), \((-1, 0)w_E= \varepsilon, (0, 1)w_E=b\), and \((1, 2)w_E=bab\).
We claim that the map \(w_0:\Z_{[-\infty, \infty]}\to D\) defined by:
\begin{align*}
    (z)(w_0)_V&= \bullet \text{ for all }z\in \Z\\
    (z, z+1)(w_0)_E&= a \text{ for all }z<0\text{ and for all }z>3\\
    (0, 1)(w_0)_E&= b, \quad (1, 2)(w_0)_E= b, \quad (2, 3)(w_0)_E= a,\quad (3, 4)(w_0)_E= b.
\end{align*}
is a flattening of \(w\).
Intuitively as we think of \(w\) as
\[\ldots (a)(a)(a)(a)()(b)(bab)(a)(a) \ldots,\]
and we think of \(w_0\) as
\[\ldots aaaabbabaa \ldots.\]
The map \(f_{w, w_0}\) is then given by:
\[(z)f_{p, p_0}= \left\{\begin{array}{cc}
z +1     &  z <-1\\
0 & z=-1\\
0    & z=0\\
1 & z= 1\\
4 & z= 2\\
z+2 & z>2
\end{array}\right\}.\]
\end{example}
Unless stated otherwise, we will always use  the notation \(\langle X\rangle\) to mean the group generated by \(X\) as opposed to semigroup/monoid.
\begin{defn}
We define \(\sigma:\Z_{[-\infty, \infty]} \to \Z_{[-\infty, \infty]}\) to be the unique automorphism with \((0)\sigma_V = 1\).
Note that \(\Aut(\Z_{[-\infty, \infty]}) = \langle \sigma \rangle\).
\end{defn}
\begin{remark}\label{flattening set}
The flattenings of a generalised bi-infinite walk are all translations of each other.
Thus if \(w\) is a non-degenerate bi-infinite walk in a digraph \(D\) and \(w_0\) is any flattening of \(w\), then
\[\Flt(w)=\langle \sigma \rangle\circ w_0.\]
In particular if \((X_D, \sigma_D)\) is the subshift of finite type corresponding to a digraph \(D\), then the flattening sets are precisely the orbits of \(\sigma_D\).
\end{remark}
\begin{defn}
For a digraph \(D\), we define an equivalence relation \(\sim_F\) on the class of all non-degenerate bi-infinite walks in digraphs by
\[p\sim_F q \iff \Flt(p)= \Flt(q).\]
We call an equivalence class \([p]_{\sim_F}\) of \(\sim_F\) a \textit{route}. 
Moreover if \(D\) is a digraph, then define 
\[\rt(D): =\makeset{[w]_{\sim_F}}{\(w\in \GBwlk(D)\)}.\]
\end{defn}
Note that \([w]_{\sim_{F}}\) and \(\Flt(w)\) carry the same information but one consists of generalised bi-infinite walks and one consists of standard bi-infinite walks.

\begin{defn}[Degenerate homomorphisms]
If \(D_1, D_2\) are finite digraphs and \(\phi:D_1 \to \Wlk(D_2)\) is a digraph homomorphism, then we say that \(\phi\) is \textit{degenerate} if there is a non-degenerate generalised bi-infinite walk \(p\) in \(D_1\) such that \(p\circ \widetilde{\phi}\) is degenerate.
\end{defn}
Note that if \(D_1\) and \(D_2\) are finite digraphs, then \(\phi:D_1\to \Wlk(D_2)\) is degenerate if and only if \(\widetilde{\phi}\) maps a cycle of length greater than \(0\) to a cycle of length \(0\).
\begin{defn}
If \(D_1, D_2\) are finite digraphs and \(\phi:D_1 \to \Wlk(D_2)\) is a non-degenerate digraph homomorphism, then we define \(\phi_*:\rt(D_1) \to \rt(D_2)\) by
\[([p]_{\sim_F})\phi_*=[p\circ \widetilde{\phi}]_{\sim_F}.\]
Similarly if \(\phi\in \Hom(D_1, D_2)\), then we define \(\phi_*\) via the standard embedding of \(D_2\) into \(\Wlk(D_2)\).
\end{defn}

\part{The category UDAFG}
In this part we introduce the category UDAFG. In particular we give various equivalent ways of viewing the morphisms of this category (see Theorem~\ref{udaf isomorphisms theorem}).

In the next definition we define the class of digraphs we will use for the objects in our categories. 
The imposed conditions will make it easier to construct ``irrational" walks (Definition~\ref{rational walks}).
Ideally we would be able to consider all digraphs, but our categories do still cover most digraphs of interest. 
In particular they allow for the \(n\)-leaf  rose \(R_n\) for all \(n\neq 1\) (the one vertex n-edge digraphs).

\begin{defn}[UDAF digraphs]\label{udaf digraph defn}
We say that a digraph \(D\) is an \textit{UDAF digraph} if:
\begin{enumerate}
    \item \(D\) is finite (in terms of both edges and vertices).
    \item For all \(v\in V_D\) we have
    \[\left|\makeset{\phi\in \Hom(\Z_{[0,\infty]},D)}{\((0)\phi_V=v\)}\right|\neq 1.\]
    \item For all \(v\in V_D\) we have
    \[\left|\makeset{\phi\in \Hom(\Z_{[-\infty,0]}, D)}{\((0)\phi_V=v\)}\right|\neq 1.\]
\end{enumerate}
\end{defn}
Before we define UDAF foldings we define a folding as given by Jim Belk, Collin Bleak, and Peter Cameron.
\begin{defn}[Foldings]\label{normal foldings defn}
Let \(D_1, D_2\) be finite digraphs.
A digraph homomorphism \(\phi:D_1\to D_2\) is called a \textit{folding} if the map \(f\to f\circ \phi\) from \(\Hom(\Z, D_1)\to \Hom(\Z, D_2)\) is a bijection.

In this case, this induced map is a topological conjugacy between subshifts of finite type.
\end{defn}
\begin{defn}[UDAF foldings and UDAFG]
Let \(D_1, D_2\) be finite digraphs.
We say that \(\phi:D_1\to \Wlk(D_2)\) is an \textit{UDAF folding} if \(\phi\) is a digraph homomorphism and \(\phi_*:\rt(D_1)\to \rt(D_2)\) is a bijection. 
We also extend this definition to elements of \(\Hom(D_1, D_2)\) (via the identification of \(D_2\) with its embedded copy in \(\Wlk(D_2)\)). In particular we include usual foldings as UDAF foldings.

We define UDAF to be the category of UDAF digraphs and UDAF foldings. 
We define UDAFG to be the smallest category such that:
\begin{enumerate}
    \item The objects of UDAFG are the pairs \((\rt(D), D)\) where \(D\) is an UDAF digraph.
    \item When \((\rt(D_1), D_1)\) to \((\rt(D_2), D_2)\) are objects of UDAFG and \(\phi:D_1 \to \Wlk(D_2)\) is an UDAF folding, the bijection \(\phi_*\) is a morphism from \((\rt(D_1), D_1)\) to \((\rt(D_2), D_2)\).
    \item The inverse of a morphism is always a morphism and the composition of two morphisms is a morphism.
\end{enumerate}
We technically have an issue here, as it is possible for a single bijection \(\phi_*\) to have multiple possible domain (or target) digraphs.
We could fix this problem by defining a morphism to instead be a triple \((D_1,\phi_*, D_2)\) but we do not do this explicitly to avoid clunky notation.
However it will be important to distinguish objects \((\rt(D_1), D_1)\) and \((\rt(D_2), D_2)\) for which \(D_1\) and \(D_2\) differ but \(\rt(D_1)\) and \(\rt(D_2)\) coincide. 
Note that all the morphisms of the category UDAFG are isomorphisms (the G stands for groupoid).
\end{defn}

The following definition is of primary importance due to the fact that if \(\phi:G_1\to G_2\) is any folding (Definition~\ref{normal foldings defn}), then we can find a past-future digraph \(P\) and a folding \(f:P\to G_1\) such that \(f\circ \phi\) is the standard folding. 
This is follows as the induced map from \(\Hom(\Z, D_1)\to \Hom(\Z, D_2)\) is a continuous bijection between compact Hausdorff spaces and hence has a continuous inverse.
In particular all foldings can be thought of in terms of these foldings. 
We will later show that a similar fact holds for UDAF foldings (Lemma~\ref{big lem}). 

The following digraphs are notionally equivalent to ``De Bruijin graphs" or iterated ``line graphs".
\begin{defn}[Past-future digraphs]\label{past-future def}
If \(D\) is a digraph, \((m, n) \in {\Z_{\leq 0}} \times \Z_{\geq 0}\) then we define its \((m, n)-\)past-future digraph \(\PF(D, m, n)\), to be the digraph with
\[V_{\PF(D, m, n)} := \Hom(\Z_{[m, n]}, D),\quad E_{\PF(D, m, n)} := \Hom(\Z_{[m, n + 1]}, D),\]
and for all \(e\in  E_{\PF(D, m, n)}\)
\[(e)s_{\PF(D, m, n)}= e\restriction_{\Z_{[m, n]}} \quad \text{ and }(e)t_{\PF(D, m, n)}= (\sigma \circ e)\restriction_{\Z_{[m, n]}}.\]
(The symbol \(\restriction\) here denotes the restriction of a digraph homomorphism to a subdigraph.)
We also define the \textit{standard folding of }\(\PF(D, m, n)\) to be the digraph homomorphism \(f:\PF(D, m, n) \to D\) with \((h)f_V= (0)h_V\) and \((h)f_E= (0, 1)h_E\).
It is routine to check that this digraph homomorphism is always a folding.

We call \(m\) and \(n\) the past and future information of \(\PF(D, m, n)\) respectively. Note that if \(m'\geq m\), \(n'\geq n\) and \(\phi_{m, n}:\PF(D, m, n)\to D\), \(\phi_{m', n'}:\PF(D, m', n')\to D\) are the standard foldings, then there is an innate folding \(f:\PF(D, m', n')\to \PF(D, m, n)\) with \(f\circ \phi_{m, n}=\phi_{m',n'}\) defined by taking the appropriate restrictions of the walks constituting the vertices and edges.
\end{defn}

The following definition gives some easy examples of UDAF foldings, we will later show that such foldings are in fact sufficient to define the entire UDAFG category (Theorem~\ref{udaf isomorphisms theorem}).
\begin{defn}[Elementary UDAF foldings]\label{elementary UDAF foldings defn}
If \(D_1, D_2\) are UDAF digraphs and \(f:D_1\to \Wlk(D_2)\) is an UDAF folding, then we say that \(f\) is an \textit{elementary} UDAF folding if one of the following holds:
\begin{enumerate}
    \item \(f\) is a digraph homomorphism which embeds \(D_1\) into \(D_2\) and such that \(D_2\) consists of the image of \(f\) and one other vertex with no edges connected to it,
    \item \(f:D_1\to \Wlk(D_2)\) is a digraph homomorphism and there is a non-loop edge \(e\in E_{D_2}\) and a subdigraph \(S=(V_{D_1}, E_S)\) of \(D_1\) such that \(f\restriction_S\) is an isomorphism from \(S\) to \((V_{D_2}, E_{D_2}\backslash \{e\})\) (here a ``loop edge" is an edge from a vertex to itself).
    Moreover \(f_E\restriction_{E_{D_1}\backslash E_S}\) is a bijection onto the set
    \[\{p\in \Hom(\Z_{[0,2]}, D_2):(0, 1)p_E= e\}\] (we call this a push from \(e\)) see Figure~\ref{push figure},
    \item \(f:D_1\to \Wlk(D_2)\) is a digraph homomorphism and there is a non-loop edge \(e\in E_{D_2}\) and a subdigraph \(S=(V_{D_1}, E_S)\) of \(D_1\) such that \(f\restriction_S\) is an isomorphism from \(S\) to \((V_{D_2}, E_{D_2}\backslash \{e\})\).
    Moreover \(f_E\restriction_{E_{D_1}\backslash E_S}\) is a bijection onto the set
    \[\{p\in \Hom(\Z_{[0,2]}, D_2):(1, 2)p_E= e\}\] (we call this a full from \(e\)) see Figure~\ref{pull figure},
\end{enumerate}
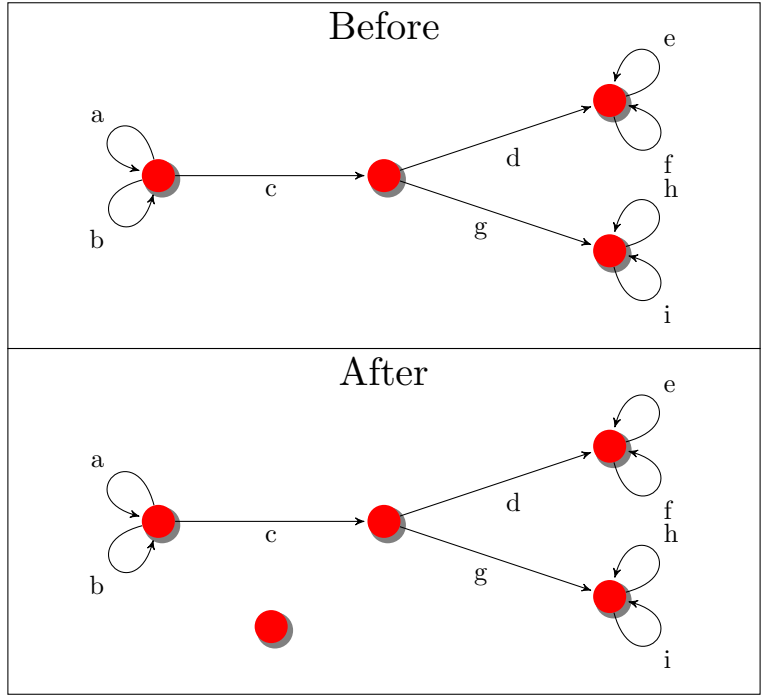
\begin{figure}
\begin{center}
\begin{tikzpicture}[->,>=stealth',shorten >=1pt,auto,node distance=5cm,on grid, every state/.style={fill=red,draw=none,circular drop shadow,text=white}]

  \node [state, scale = 0.5] (A) at (-3,0) {};
  \node [state, scale = 0.5] (B) at (0,0) {};
  \node [state, scale = 0.5] (C) at (3, 1) {};
  \node [state, scale = 0.5] (D) at (3, -1) {};
  \node [scale = 1.5] (T) at (0, 2) {Before};
  \draw (-5, -2.3) rectangle (5, 2.3);
  
    \node [scale = 1.5] (T2) at (0, -2.6) {After};
  \node [state, scale = 0.5] (A2) at (-3,-4.6) {};
  \node [state, scale = 0.5] (B2) at (0,-4.6) {};
  \node [state, scale = 0.5] (C2) at (3, -3.6) {};
  \node [state, scale = 0.5] (D2) at (3, -5.6) {};
  \node [state, scale = 0.5] (X) at (-1.5, -6) {};
  \draw (-5, -6.9) rectangle (5, -2.3);

 \path [->]
 (A) edge  [out=105,in=165, loop, scale = 2]
 node  [swap]{a} (A)
 (A) edge  [out=195,in=255, loop, scale = 2]
 node [swap]{b} (A)
 (A) edge  [out=0,in=180, scale = 2]
 node [swap]{c} (B)
 (B) edge  [ scale = 2]
 node [swap]{d} (C)
 (C) edge  [out=15,in=75, loop, scale = 2]
 node [swap]{e} (C)
 (C) edge  [out=285,in=345, loop, scale = 2]
 node  [swap]{f} (C)
 (B) edge  [ scale = 2]
 node [swap]{g} (D)
 (D) edge  [out=15,in=75, loop, scale = 2]
 node [swap]{h} (D)
 (D) edge  [out=285,in=345, loop, scale = 2]
 node  [swap]{i} (D)
 
 (A2) edge  [out=105,in=165, loop, scale = 2]
 node  [swap]{a} (A2)
 (A2) edge  [out=195,in=255, loop, scale = 2]
 node [swap]{b} (A2)
 (A2) edge  [out=0,in=180, scale = 2]
 node [swap]{c} (B2)
 (B2) edge  [ scale = 2]
 node [swap]{d} (C2)
 (C2) edge  [out=15,in=75, loop, scale = 2]
 node [swap]{e} (C2)
 (C2) edge  [out=285,in=345, loop, scale = 2]
 node  [swap]{f} (C2)
 (B2) edge  [ scale = 2]
 node [swap]{g} (D2)
 (D2) edge  [out=15,in=75, loop, scale = 2]
 node [swap]{h} (D2)
 (D2) edge  [out=285,in=345, loop, scale = 2]
 node  [swap]{i} (D2)
 ;
 \end{tikzpicture}
 \end{center}
\caption{An example of the effect of an elementary folding of the first type}\label{baker}
\end{figure}
\begin{figure}
\begin{center}
\begin{tikzpicture}[->,>=stealth',shorten >=1pt,auto,node distance=5cm,on grid, every state/.style={fill=red,draw=none,circular drop shadow,text=white}]

  \node [state, scale = 0.5] (A) at (-3,0) {};
  \node [state, scale = 0.5] (B) at (0,0) {};
  \node [state, scale = 0.5] (C) at (3, 1) {};
  \node [state, scale = 0.5] (D) at (3, -1) {};
  \node [scale = 1.5] (T) at (0, 2) {Before};
  \draw (-5, -2.3) rectangle (5, 2.3);
  
    \node [scale = 1.5] (T2) at (0, -2.6) {After};
  \node [state, scale = 0.5] (A2) at (-3,-4.6) {};
  \node [state, scale = 0.5] (B2) at (0,-4.6) {};
  \node [state, scale = 0.5] (C2) at (3, -3.6) {};
  \node [state, scale = 0.5] (D2) at (3, -5.6) {};
  \draw (-5, -6.9) rectangle (5, -2.3);

 \path [->]
 (A) edge  [out=105,in=165, loop, scale = 2]
 node  [swap]{a} (A)
 (A) edge  [out=195,in=255, loop, scale = 2]
 node [swap]{b} (A)
 (A) edge  [out=0,in=180, scale = 2]
 node [swap]{c} (B)
 (B) edge  [ scale = 2]
 node [swap]{d} (C)
 (C) edge  [out=15,in=75, loop, scale = 2]
 node [swap]{e} (C)
 (C) edge  [out=285,in=345, loop, scale = 2]
 node  [swap]{f} (C)
 (B) edge  [ scale = 2]
 node [swap]{g} (D)
 (D) edge  [out=15,in=75, loop, scale = 2]
 node [swap]{h} (D)
 (D) edge  [out=285,in=345, loop, scale = 2]
 node  [swap]{i} (D)
 
 (A2) edge  [out=105,in=165, loop, scale = 2]
 node  [swap]{a} (A2)
 (A2) edge  [out=195,in=255, loop, scale = 2]
 node [swap]{b} (A2)
 (A2) edge  [ scale = 2]
 node []{cd} (C2)
 (A2) edge  [scale = 2]
 node [swap]{cg} (D2)
 (B2) edge  [ scale = 2]
 node [swap]{d} (C2)
 (C2) edge  [out=15,in=75, loop, scale = 2]
 node [swap]{e} (C2)
 (C2) edge  [out=285,in=345, loop, scale = 2]
 node  [swap]{f} (C2)
 (B2) edge  [ scale = 2]
 node []{g} (D2)
 (D2) edge  [out=15,in=75, loop, scale = 2]
 node [swap]{h} (D2)
 (D2) edge  [out=285,in=345, loop, scale = 2]
 node  [swap]{i} (D2)
 ;
 \end{tikzpicture}
 \end{center}
\caption{An example of the effect of an elementary folding of the second type (pushing from c).}\label{push figure}
\end{figure}
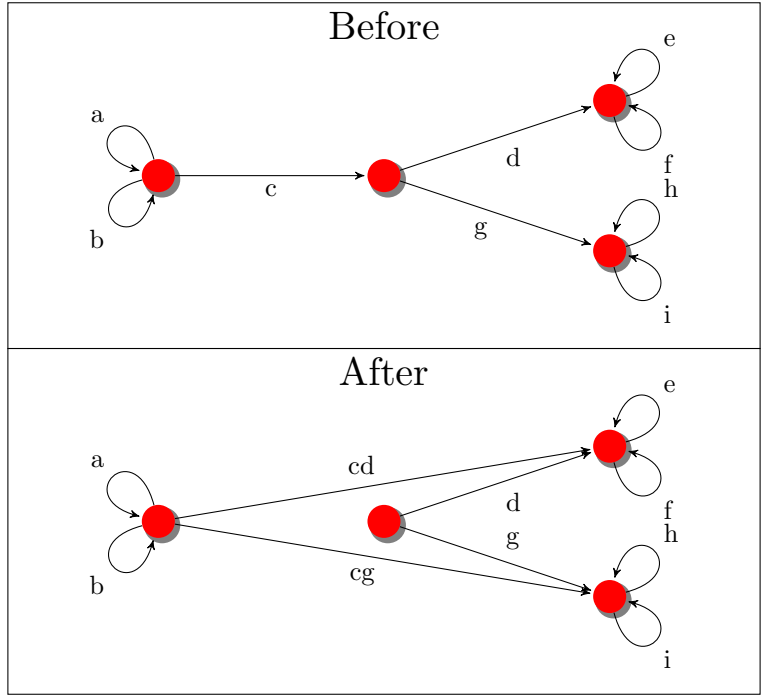
\begin{figure}
\begin{center}

\begin{tikzpicture}[->,>=stealth',shorten >=1pt,auto,node distance=5cm,on grid, every state/.style={fill=red,draw=none,circular drop shadow,text=white}]

  \node [state, scale = 0.5] (A) at (-3,0) {};
  \node [state, scale = 0.5] (B) at (0,0) {};
  \node [state, scale = 0.5] (C) at (3, 1) {};
  \node [state, scale = 0.5] (D) at (3, -1) {};
  \node [scale = 1.5] (T) at (0, 2) {Before};
  \draw (-5, -2.3) rectangle (5, 2.3);
  
    \node [scale = 1.5] (T2) at (0, -2.6) {After};
  \node [state, scale = 0.5] (A2) at (-3,-4.6) {};
  \node [state, scale = 0.5] (B2) at (0,-4.6) {};
  \node [state, scale = 0.5] (C2) at (3, -3.6) {};
  \node [state, scale = 0.5] (D2) at (3, -5.6) {};
  \draw (-5, -6.9) rectangle (5, -2.3);

 \path [->]
 (A) edge  [out=105,in=165, loop, scale = 2]
 node  [swap]{a} (A)
 (A) edge  [out=195,in=255, loop, scale = 2]
 node [swap]{b} (A)
 (A) edge  [out=0,in=180, scale = 2]
 node [swap]{c} (B)
 (B) edge  [ scale = 2]
 node [swap]{d} (C)
 (C) edge  [out=15,in=75, loop, scale = 2]
 node [swap]{e} (C)
 (C) edge  [out=285,in=345, loop, scale = 2]
 node  [swap]{f} (C)
 (B) edge  [ scale = 2]
 node [swap]{g} (D)
 (D) edge  [out=15,in=75, loop, scale = 2]
 node [swap]{h} (D)
 (D) edge  [out=285,in=345, loop, scale = 2]
 node  [swap]{i} (D)
 
 (A2) edge  [out=105,in=165, loop, scale = 2]
 node  [swap]{a} (A2)
 (A2) edge  [out=195,in=255, loop, scale = 2]
 node [swap]{b} (A2)
 (A2) edge  [out=15,in=165, scale = 2]
 node []{ac} (B2)
 (A2) edge  [out=345,in=195, scale = 2]
 node [swap]{bc} (B2)
 (B2) edge  [ scale = 2]
 node [swap]{d} (C2)
 (C2) edge  [out=15,in=75, loop, scale = 2]
 node [swap]{e} (C2)
 (C2) edge  [out=285,in=345, loop, scale = 2]
 node  [swap]{f} (C2)
 (B2) edge  [ scale = 2]
 node [swap]{g} (D2)
 (D2) edge  [out=15,in=75, loop, scale = 2]
 node [swap]{h} (D2)
 (D2) edge  [out=285,in=345, loop, scale = 2]
 node  [swap]{i} (D2)
 ;
 \end{tikzpicture}
 \end{center}
 
\caption{An example of the effect of an elementary folding of the third type (pulling from c).}\label{pull figure}
\end{figure}
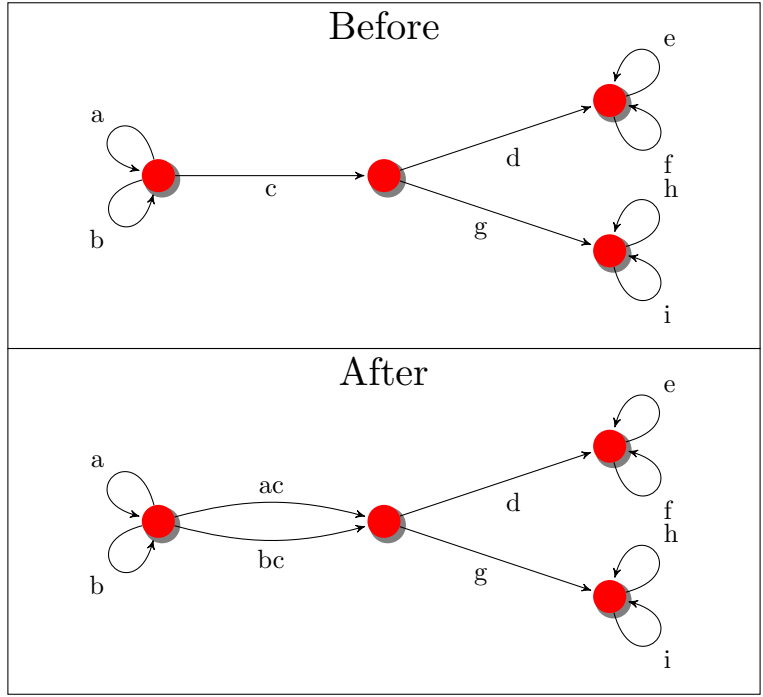

It is routine to verify that the above maps are indeed UDAF foldings.
As the domain digraph \(D_1\) of a pull/push is usually larger than the target digraph \(D_2\), we will we thinking of a pull/push as a operation to replace the digraph \(D_2\) with the digraph \(D_1\) (for example we will say things like ``convert \(D_2\) to \(D_1\) via a push from \(e\)").
\end{defn}

\begin{defn}
We define UDAFG' to be the subcategory of UDAFG generated by the maps \(\phi_*\) and \((\phi_*)^{-1}\) where \(\phi\) is an elementary UDAFG folding between UDAF digraphs.
\end{defn}
We will later show that UDAFG' and UDAFG are actually the same (Theorem~\ref{big groupoid}).

\section{Generation by elementary foldings}
In this section we establish some basic facts about elementary UDAF foldings and show that the categories UDAFG and UDAFG' do indeed coincide (Corollary~\ref{big groupoid}).

We first consider splittings which by the Decomposition Theorem (see Theorem \ref{splitings are nice theorem}) is enough to cover strong shift equivalence.
We then conclude the argument with one of the key lemmas of the document (Lemma~\ref{big lem}).

\begin{defn}[Splittings]
Suppose that \(D\) is a finite digraph and \(\sim\) is an equivalence relation on \(E_D\) such that \(e_1\sim e_2 \Rightarrow (e_1)s_D= (e_2)s_D\).

We define the corresponding \textit{out-splitting} digraph \(\operatorname{OS}(D, \sim)\) of \(D\) with respect to \(\sim\) to be the digraph \(S\) with 
\[V_S= \makeset{[e]_\sim}{\(e\in E_D\)}\quad \text{ and }\quad E_S= \makeset{(e_1, [e_2]_\sim)\in E_D\times V_S}{\((e_1)t_D=(e_2)s_D\)}.\]
Where \((e_1, [e_2]_\sim)\) is an edge from \([e_1]_\sim\) to \([e_2]_\sim\).
We also define the \textit{standard folding} \(f:S\to D\) by
\[([e]_\sim)f_V = (e)s_D \quad\text{ and }\quad (e_1, [e_2]_\sim)f_E = e_1.\]
It is routine to verify that this is a folding.
If \(\sim\) is an equivalence relation on \(E_D\) such that \(e_1\sim e_2 \Rightarrow (e_1)t_D= (e_2)t_D\), then we define the \textit{in-splitting} digraph \(\operatorname{IS}(D, \sim)\) analogously.
\end{defn}

\begin{theorem}[Decomposition Theorem, c.f. Theorem 7.1.2 of \cite{lind2021introduction}]\label{splitings are nice theorem}
Every topological conjugacy between subshifts of finite type can be obtained via a sequence of splittings and their inverses.
\end{theorem}

The notion of irrationality is our main tool for keeping track of unindexed walks (hence the need to restrict the class of UDAF digraphs).
\begin{defn}[Rationality]\label{rational walks}
If \(w:\Z_{[0, \infty]} \to D\) is a digraph homomorphism, then we call \(w\) an \textit{infinite forwards walk}.
Similarly, if \(w:\Z_{[-\infty, 0]} \to D\) is a digraph homomorphism, then we call \(w\) an \textit{infinite backwards walk}.
We say that an infinite forwards walk \(w\) is \textit{rational} if there are \(n, m\in \Z_{\geq 0}\) such that for all \(i\geq n\) we have \[(i)w_V=(i+m)w_V\quad \text{ and } \quad (i, i + 1)w_E=(i+m, i+1 +m)w_E.\]

Similarly if \(w\) is an infinite backwards walk, then we say that \(w\) is \textit{rational} if there are \(n, m\in \Z_{\leq 0}\) such that for all \(i< n\) we have 
\[(i)p_V=(i+m)p_V \quad \text{ and }\quad(i, i + 1)p_E=(i+m, i+1 +m)p_E.\]
If \(p:\Z_{[-\infty, \infty]} \to D\) is a digraph homomorphism, then we say that \(p\) is \textit{rational} if both \(p\restriction_{\Z_{[-\infty, 0]}}\) and \(p\restriction_{\Z_{[0, \infty]}}\) are rational. 
We refer to any of these as \textit{irrational} if they are not rational.
\end{defn}
\begin{theorem}\label{elsplit}
If \(D\) is a finite digraph, \(S\) is an in-splitting or out-splitting of \(D\) and \(f:S\to D\) is the standard folding, then \(f_*\) is an UDAFG' isomorphism.
\end{theorem}
\begin{proof}
We assume without loss of generality that \(S= \operatorname{OS}(D, \sim)\) for some equivalence relation \(\sim\). 

Let \(D_0:= D\).
Let \(D_1\) be the digraph obtained from \(D_0\) by adding \(V_S\) to the vertex set of \(D_0\) and let \(\phi_0:D_0 \to D_1\) be the embedding.
As \(\phi_0\) is a composition of elementary UDAF foldings of the first type, it follows that \({\phi_0}_*\) is an UDAFG' isomorphism. 

 Let \(D_2\) be the digraph obtained from \(D_1\) by adding the set
 \[\makeset{([e]_\sim, e, (e)t_D)}{\(e\in E_D\)}\]
to the edge set of \(D_1\) (\(([e]_\sim, e, (e)t_D)\) is an edge from the vertex \([e]_\sim\) of \(S\) to the vertex \((e)t_D\) of \(D_1\)).
Let \(\phi_1:D_1 \to D_2\) be the embedding.
To see that \({\phi_1}_*\) is an UDAFG' isomorphism, note that there are no edges in \(D_1\) or \(D_2\) going to any vertices in \(V_S\).
Thus the natural embedding of \(D_1\) into the digraph obtained by adding a single one of these new edges is a elementary UDAF folding of the third type (a pull from the new edge). By repeating this we can add all of the new edges.

Let \(D_3\) be the digraph obtained from \(D_2\) by removing all the edges originally from \(D\) and adding
\(\makeset{(v, [e]_\sim)}{\((e)s_D= v\)}\) to the edge set (\((v, [e]_\sim)\) is an edge from the vertex \(v\) of \(D\) to the vertex \([e]_\sim\)).
Let \(\phi_2:D_2 \to \Wlk(D_3)\) be the digraph homomorphism defined by mapping each \(e\in E_D\) to the walk of length \(2\) using the edges \(((e)s_D, [e]_\sim)\) and \(([e]_\sim, e, (e)t_D)\) (and fixing the rest of the digraph).
It is routine to check that \(\phi_2\) is a composition of elementary UDAF foldings of the second type (pushing from the new edges of \(D_3\)).
Thus \({\phi_2}_*\) is an UDAFG' isomorphism.

Let \(D_6:= S\). Let \(D_5\) be the digraph obtained from \(D_6\) by adding \(V_D\) to the vertex set of \(D_6\) and let \(\phi_6:D_6 \to D_5\) be the embedding.
As \(\phi_6\) is a composition of elementary foldings of the first type, it follows that \({\phi_6}_*\) is an UDAFG' isomorphism.

Let \(D_4\) be the digraph obtained from \(D_5\) by adding \(\makeset{(v, [e]_\sim)}{\((e)s_D= v\)}\) to the edge set (\((v, [e]_\sim)\) is an edge from the vertex \(v\) of \(D\) to the vertex \([e]_\sim\)).

Let \(\phi_5:D_5 \to D_4\) be the embedding.
The fact that \({\phi_5}_*\) is an UDAFG' isomorphism, follows from a symmetric argument to the fact that \({\phi_1}_*\) was an UDAFG' isomorphism.

Finally let \(\phi_4:D_4 \to \Wlk(D_3)\) be the digraph homomorphism defined by mapping each \((e_1, [e_2]_\sim)\in E_S\) to the walk of length \(2\) using the edges \(([e_1]_\sim, e_1, (e_1)t_D)\) and \(((e_1)t_D, [e_2]_\sim)\).
The map \(\phi_4\) is then a composition of elementary UDAF foldings of the third type.

Thus \({\phi_6}_*{\phi_5}_*{\phi_4}_*{\phi_2}_*^{-1}{\phi_1}_*^{-1}{\phi_0}_*^{-1}\) is an UDAFG' isomorphism from \(\rt(S)\) to \(\rt(D)\). 
Moreover, if \((e_1, [e_2]_{\sim})\in E_S\) is arbitrary, then
\begin{align*}
    (e_1, [e_2]_{\sim}){\phi_6}\widetilde{\phi_5}\widetilde{\phi_4}=(e_1, [e_2]_{\sim})\phi_4
\end{align*}
is the walk of length \(2\) using the edges \(([e_1]_\sim, e_1, (e_1)t_D)\) and \(((e_1)t_D, [e_2]_\sim)=((e_2)s_D, [e_2]_\sim)\).
Similarly
\[(e_1, [e_2]_{\sim})f\widetilde{\phi_0}\widetilde{\phi_1}\widetilde{\phi_2}=(e_1)\phi_2\]
is the walk of length \(2\) using the edges \(((e_1)s_D, [e_1]_\sim)\) and \(([e_1]_\sim, e_1, (e_1)t_D)\).

Thus if \(w\in \rt(S)\) is arbitrary and \(\ldots  e_1 e_2\ldots\) is the bi-infinite sequence of edges corresponding to \((w)f^*\), then
\((w){\phi_6}_*{\phi_5}_*{\phi_4}_*\) and \((w)f_*{\phi_0}_*{\phi_1}_*{\phi_2}_*\) are both equal to the route
\[\ldots((e_1)s_D, [e_1]_\sim)([e_1]_\sim, e_1, (e_1)t_D)((e_2)s_D, [e_2]_\sim)([e_2]_\sim, e_2, (e_2)t_D)\ldots\]
of \(D_3\). 
Thus \(f_*={\phi_6}_*{\phi_5}_*{\phi_4}_*{\phi_2}_*^{-1}{\phi_1}_*^{-1}{\phi_0}_*^{-1}\) and the result follows.

\end{proof}

\begin{corollary}\label{SSE implies UDAF cor}
If \(D_1\) and \(D_2\) are strong shift equivalent UDAF digraphs, then the objects \((\Wlk(D_1), D_1)\) and \((\Wlk(D_2), D_2)\) are isomorphic in UDAFG'.
\end{corollary}
\begin{defn}
If \(D_1, D_2\) are digraphs, then the set \(\Hom(D_1,D_2)\) is equipped with a topology, which we will use throughout the rest of the document.
A digraph homomorphism \(p:D_1\to D_2\) is a pair of maps \((p_V, p_E)\) each of which is an element of a space of functions using the topology of pointwise convergence (with discrete topology on the targets), thus we obtain a natural topology on \(\Hom(D_1,D_2)\) (the product topology).
\end{defn}
We next define the notions of forwards and backwards images of a digraph homomorphism at a vertex. 
We also extend these notions to digraph homomorphisms to walk digraphs. 
These will parallel the images of transducer states in Section~\ref{GNS section}.
\begin{defn}\label{forwards/backwards image}
If \(D_1, D_2\) are finite digraphs, \(v\in V_{D_1}\) and \(f:D_1\to D_2\) is a digraph homomorphism, then we define the forwards and backwards images of \(v\) with respect to \(f\) by
\[\fim_f(v) := \makeset{p\in \Hom(\Z_{[0, \infty]}, D_1)}{\((0)p_V=v\)}\circ f,\]
\[\bim_f(v) := \makeset{p\in \Hom(\Z_{[-\infty, 0]}, D_1)}{\((0)p_V=v\)}\circ f.\]
Note that \(\makeset{p\in \Hom(\Z_{[0, \infty]}, D_1)}{\((0)p_V=v\)}\) is a closed subspace of a product of finite discrete spaces, and is thus compact.
Moreover post composition with \(f\) is a continuous map from \(\Hom(\Z_{[0, \infty]}, D_1)\) to \(\Hom(\Z_{[0, \infty]}, D_2)\).
It follows that the sets \(\fim_f(v)\) (and similarly \(\bim_f(v)\)) are always compact.
\end{defn}

\begin{defn}
Let \(\operatorname{bcat}:
\Hom(\Z_{[-\infty, 0]}, \Wlk(D_2))\to \Hom(\Z_{[-\infty, 0]}, D_2)\)
and \(\operatorname{fcat}:\Hom(\Z_{[0, \infty]}, \Wlk(D_2))\to \Hom(\Z_{[0,\infty]}, D_2)\) be the backwards and forwards concatenation/flattening maps respectively.
That is if \(w\in \Hom(\Z_{[0, \infty]}, \Wlk(D_2))\) and for all \(n\geq 0\) and \(j< \norm{(n, n+1)w_E}\) we define \(m_{n, j}:=j+\sum_{i<n}\norm{(i, i+1)w_E}\), then \((w)\operatorname{fcat}\) is the digraph homomorphism with
\[(m_{n,j})((w)\operatorname{fcat})_V=(j)((n, n+1)w_E)_V\]
and
\[(m_{n,j}, m_{n,j}+1)((w)\operatorname{fcat})_E=(j, j+1)((n, n+1)w_E)_E\]
(\(\operatorname{bcat}\) is defined analogously).

Now if \(f:D_1\to \Wlk(D_2)\) is a digraph homomorphism, then we extend the notions of forwards and backwards images to \(f\) by defining
\[\fim_f(v) := \left(\makeset{p\in \Hom(\Z_{[0, \infty]}, D_1)}{\((0)p_V=v\)}\circ f \right) \operatorname{fcat},\]
\[\bim_f(v) := \left(\makeset{p\in \Hom(\Z_{[-\infty, 0]}, D_1)}{\((0)p_V=v\)}\circ f\right) \operatorname{bcat}.\]
\end{defn}

The proof of the following lemma is highly technical but it is perhaps the most important lemma in this document.
It allows us to carry over many of the useful properties of foldings to UDAF foldings. 
\begin{lemma}\label{big lem}
If \(D_1, D_2\) are UDAF digraphs and \(f:D_1\to \Wlk(D_2)\) is an UDAF folding, then the following hold:
\begin{enumerate}
    \item There is a past-future digraph \(P\) with standard folding \(h:P\to D_2\), and an UDAF folding \(g:P\to \Wlk(D_1)\) with \(h_*=g_*f_*\) and such that \(g_*\) is an UDAFG' isomorphism.
    \item If \(v\in V_{D_1}\), then the sets \(\fim_f(v)\) and \(\bim_f(v)\) are clopen subsets of \(\Hom(\Z_{[0,\infty]}, D_2)\) and \(\Hom(\Z_{[-\infty, 0]}, D_2)\) respectively.
\end{enumerate}
\end{lemma}
\begin{proof}
To show the first point, we will sequentially modify \(D_1\) and \(f\) using UDAF foldings which induce UDAFG' isomorphisms so that \(f\)  will induce a bijection between \(\Hom(\Z_{[-\infty, \infty]},D_1)\) and \(\Hom(\Z_{[-\infty, \infty]}, D_2)\).
This is sufficient as then \(f\) will be a folding (as in Definition~\ref{normal foldings defn}) and hence the domain of \(f\) can be appropriately converted into a past-future digraph of \(D_2\) using splittings (recall Theorem~\ref{splitings are nice theorem} and the discussion of past-future digraphs). 
In particular, the required \(g\) will be the composite of a standard folding from a past-future digraph of \(D_2\) to the new domain of \(f\) with the UDAF foldings we will use to modify \(f\).
In order to work towards the second point of the lemma, we will also be modifying \(D_1\) and \(f\) in such as way that \(D_1\) never loses any of its original vertices and for each \(v\in V_{D_1}\), the set \(\fim_f(v)\) does not change.
The statement we want to prove about \(\bim_f(v)\) is symmetric with the statement about \(\fim_f(v)\), so we will only consider \(\fim_f(v)\) (indeed our changes to \(f\) may affect \(\bim_f(v)\)).

If \(f\) maps an edge to a length \(0\) walk, then there is an edge \(e_1\) which is mapped to a length \(0\) walk (by \(f\)) such that all edges \(e_2\) with \((e_1)t_{D_1}=(e_2)s_{D_1}\) are not mapped to length \(0\) walks (this follows as otherwise \(f\) would be degenerate).
Thus by precomposing \(f\) with a push from this edge \(e_1\), we can reduce the number of edges mapping to length \(0\) walks by \(1\).
Thus by repeating this, we can assume without loss of generality that \(f\) maps every edge to a walk of positive length. 
Note that this process does not change \(\fim_f(v)\) for any vertex \(v\) (it might change \(\bim_f(v)\)). 

If \(f\) maps an edge \(e\) to a walk of length greater than \(1\), then we make the following modifications to \(D_1\).
We add a new vertex \(v_e\) and add an edge from the start of \(e\) to \(v_e\) (this can be done with the inverse of an elementary UDAF folding of type 1, followed by the inverse of a push from the new edge).
We then replace the edge \(e\) with an edge from \(v_e\) to the end of \(e\) (this can be done via a the inverse of a push from the edge going to \(v_e\)).
The result of this paragraph is an UDAFG' isomorphism from a digraph \(D_1'\) to \(D_1\) where \(D_1'\) is obtained from \(D_1\) by replacing of the edge \(e\) by a length \(2\) walk.

We can find an UDAF folding to induce this UDAFG' isomorphism by fixing the shared parts of \(D_1'\) and \(D_1\), mapping the edge going to \(v_e\) to \(e\) and mapping the other edge to the walk of length \(0\) at \((e)t_{D_1}\).
Note that this process does not change \(\fim_f(v)\) for any vertex \(v\).

By repeated applications of this process we can replace \(e\) with a path of the same length as \((e)f\).
The new \(\tilde{f}\) will now map these paths to the same walks that \(f\) mapped the original edge \(e\) to.
By repeatedly applying this process we may assume without loss of generality that for every edge \(e\) of \(D_1\) which is mapped (by \(f\)) to a walk of positive length, \(e\) is the start of a path of length \(\norm{(e)f}\) in \(D_1\) (with no other edges connecting to its internal vertices) for which all the other edges in the path are mapped to length \(0\) walks. 
Moreover from before we may assume that all edges mapped to length \(0\) walks are of the above type.
It follows that we can define a unique digraph homomorphism \(f':D_1 \to D_2\) by the equalities \((p)\tilde{f'}=(p)\tilde{f}\) for each path \(p\) of the type above.
In particular \(f'_*=f_*\) and for each vertex \(v\) of \(D_1\) there is a vertex \(v'\) of \(D_1\) such that the forwards image of \(f\) at \(v\) is equal to the forwards image of \(f'\) at \(v'\) (if \(v\) in the interior of one of the paths constructed above then we can choose \(v'\) to be the vertex at the end of the path). 
We may therefore assume without loss of generality that \(f\) is a digraph homomorphism from \(D_1\) to \(D_2\).

As \(f:D_1\to D_2\) is a digraph homomorphism, we have that \(\Hom(\Z_{[-\infty, \infty]}, D_1) f\subseteq \Hom(\Z_{[-\infty, \infty]}, D_2)\). 
As \(f_*\) is a surjection, it follows from Remark~\ref{flattening set} that the map \(w\mapsto wf\) from \(\Hom(
\Z_{[-\infty, \infty]}, D_1)\) to \(\Hom(
\Z_{[-\infty, \infty]}, D_2)\) is a surjection.
It remains to show that it is an injection and that the sets \(\fim_f(v)\) are clopen.

\underline{Claim} For every vertex \(v\in D_1\), the set \(\fim_f(v)\) is clopen.\\
\underline{Proof of Claim}:
As mentioned in Definition~\ref{forwards/backwards image}, the sets \(\fim_f(v)\) are closed, so we will now show that they are open. 
It suffices to consider the vertices of \(D_1\) which are involved in bi-infinite walks, as the forwards walks from any vertex will always reach circuits.
Let \(v\) be such a vertex of \(D_1\). 
As \(D_1\) is an UDAF digraph and \(v\) is involved in a bi-infinite walk, we can find an irrational infinite backwards walk \(w:\Z_{[-\infty, 0]}\to D_1\) with \((0)w_V=v\) (recall Definitions~\ref{udaf digraph defn} and \ref{rational walks}).

Note that
\begin{align*}
    \makeset{s\in \Hom(\Z_{[0, \infty]}, D_2)}{\((0)s_V=(v)f_V\)}
    &=\makeset{x\restriction_{\Z_{[0, \infty]}}}{\(x\in \Hom(\Z_{[-\infty, \infty]}, D_2)\)\\
    and \(x\restriction_{\Z_{[-\infty, 0]}}=wf\)}\\
    &=\bigcup_{v':wf\in \bim_f(v')} {\fim_f(v')}.
\end{align*}
Thus to show that \(\fim_f(v)\) is open, it suffices to show that the forwards images of the \(v'\) vertices (other than \(v\)) in the above union are disjoint from the forwards image of \(v\) (as then the complement of \(\fim_f(v)\) will be closed).

Suppose for a contradiction that \(a, b\in \Hom(\Z_{[0, \infty]}, D_1)\) with \((0)a_V=v\), \((0)b_V=v'\neq v\), \(wf\in \bim_f(v')\) and \(af=bf\). 
As \(wf\in \bim_f(v')\), we can find an infinite backwards walk \(w_2\in \Hom(\Z_{[-\infty, 0]}, D_1)\) with
\((0){w_2}_V=v'\) and \(w_2f=wf\).
Let \(w\star a:\Z_{[-\infty, \infty]}\to D_1\) be the bi-infinite walk which has both \(w\) and \(a\) as restrictions, and define \(w_2\star b\) analogously.

Thus by definition we have \((w\star a)f = (w_2\star b)f\). 
So by Remark~\ref{flattening set} (and the fact that \(f_*\) is a bijection), there is \(n\in \Z\) such that \(w\star a= \sigma^n\circ (w_2\star b)\). 
As \(v\neq v'\), \(w\star a \neq w_2\star b\) and so \(n\neq 0\). 
Thus as
\[(w\star a)\circ f=\sigma^n \circ (w_2\star b)\circ f=\sigma^n\circ (w\star a)\circ f,\]
it follows that \((w\star a)\circ f\) is rational. 
Let \(k_1< k_2\) be multiples of \(n\) such that \((k_1n)(w\star a)_V= (k_2n)(w\star a)_V\).
It follows that the bi-infinite walk \(r\) in \(D_1\) which repeats \((w\star a)\restriction_{\Z_{[k_1n, k_2n]}}\) infinitely, is mapped to the same bi-infinite walk (by \(f\)) as the \(w\star a\).
Hence as \(f_*\) is a bijection, it follows that \(w\star a\in \langle\sigma\rangle \circ r\) is rational.
This is a contradiction as \(w\) was chosen to be irrational.\(\diamondsuit\)\\

We can now finally show that the map \(w\mapsto wf\) from \(\Hom(
\Z_{[-\infty, \infty]}, D_1)\) to \(\Hom(
\Z_{[-\infty, \infty]}, D_2)\) is injective.
Let \(w_1, w_2\in \Hom(
\Z_{[-\infty, \infty]}, D_1)\) be arbitrary with \(w_1f=w_2f\). Let \(z\in \Z\) be arbitrary. It suffices to show that \((z-1, z){w_1}_E= (z-1, z){w_2}_E\).
Note that the set \(\fim_f((z){w_1}_V)\cap \fim_f((z){w_2}_V)\) is open (it is also non-empty). 
Thus we can find \(w_1', w_2'\in \Hom(\Z_{[-\infty, -\infty]}, D_1)\) such that
\[w_1\restriction_{\Z_{[-\infty,z]}}=w_1'\restriction_{\Z_{[-\infty,z]}},\quad w_2\restriction_{\Z_{[-\infty,z]}}=w_2'\restriction_{\Z_{[-\infty,z]}}, \quad w_1'f=w_2'f\]
and \((w_1'f)\restriction_{\Z_{[z, \infty]}}=(w_2'f)\restriction_{\Z_{[z, \infty]}}\) is irrational.

It follows that there is \(n\in \Z\) with \(w_1'=\sigma^n w_2'\). Thus \((w_1'f)\restriction_{\Z_{[z, \infty]}}=(\sigma^nw_2'f)\restriction_{\Z_{[z, \infty]}}\). As this walk is irrational, we must have \(n=0\), so 
\[(z-1, z){w_1}_E=(z-1, z){w_1'}_E= (z-1, z){\sigma^n w_2'}_E=(z-1, z){w_2'}_E=(z-1, z){w_2}_E\]
as required.

\end{proof}
\begin{corollary}\label{big groupoid}
UDAFG=UDAFG'.
\end{corollary}
\begin{proof}
Lemma~\ref{big lem} and Theorem~\ref{elsplit}.
\end{proof}
\section{Transducerable maps}
We now introduce UDAF transducers, this viewpoint of UDAF is particularly useful when connecting UDAFG isomorphisms groups to outer automorphism groups of Thompson's groups.
\begin{defn}
If \(A, B\) are UDAF digraphs and \(\phi: \rt(A) \to \rt(B)\) is a bijection, then we say that \(\phi\) is \textit{transducerable} if there are is an UDAF digraph \(T\) and UDAF foldings \(g:T \to \Wlk(A)\), \(h:T\to \Wlk(B)\) with \({g_*}^{-1}h_*= \phi\).
In this case we say that \((T, g, h)\) is an \textit{UDAF transducer} (which defines \(\phi\)). 
\end{defn}
\begin{lemma}\label{transducerable category}
Compositions of transducerable maps are transducerable.
\end{lemma}
\begin{proof}
Let \((T_1, f_1, g_1)\) and \((T_2, f_2, g_2)\) be transducers such that the target of \(g_1\) is the same as the target of \(f_2\). It suffices to show that the map \({{f_1}_*}^{-1}{g_1}_*{{f_2}_*}^{-1}{g_2}_*\) is transducerable. By Lemma \ref{big lem}, let \(h_1:P_1\to \Wlk(T_1)\) and \(h_2:P_2\to \Wlk(T_2)\) be such that \((h_1g_1)_*\) and \((h_2f_2)_*\) can be induced by standard past-future digraph UDAF foldings. 

By adding past and future information to \(P_1\) and \(P_2\) such that they coincide (see the comments in Definition~\ref{past-future def}) we may assume without loss of generality that \((h_2f_2)_*=(h_1f_1)_*\) and \(P_1=P_2\).
Thus
\begin{align*}
    {{f_1}_*}^{-1}{g_1}_*{{f_2}_*}^{-1}{g_2}_*    &= {{f_1}_*}^{-1}{{h_1}_*}^{-1}{h_1}_*{g_1}_*{{f_2}_*}^{-1}{{h_2}_*}^{-1}{h_2}_*{g_2}_*\\
     &= {{f_1}_*}^{-1}{{h_1}_*}^{-1}{h_1}_*{g_1}_*({h_2}_*{f_2}_*)^{-1}{h_2}_*{g_2}_*\\
     &= {{f_1}_*}^{-1}{{h_1}_*}^{-1}{h_1}_*{g_1}_*({h_1}_*{g_1}_*)^{-1}{h_2}_*{g_2}_*\\
     &= {{f_1}_*}^{-1}{{h_1}_*}^{-1}{h_2}_*{g_2}_*.
\end{align*}
This is then defined by the transducer \((P_1, h_1f_1, h_2g_2)\).
\end{proof}
\begin{theorem}[UDAF isomorphisms]\label{udaf isomorphisms theorem}
If \(A, B\) are UDAF digraphs and \(\phi: \rt(A)\to \rt(B)\) is a bijection, then the following are equivalent:
\begin{enumerate}
    \item \(\phi\) is an UDAFG isomorphism.
    \item \(\phi\) is an UDAFG' isomorphism.
    \item \(\phi\) is transducerable.
\end{enumerate}
\end{theorem}
\begin{proof}
\((1\iff 2):\) This follows from Corollary~\ref{big groupoid}.

\((1 \Rightarrow 3):\) By Lemma~\ref{transducerable category}, the transducerable maps form a category. This category is a groupoid as inversion can be performed by swapping the two UDAF foldings in a transducer. 
Thus it suffices to show that if \(f:D_1 \to \Wlk(D_2)\) is an UDAF folding, then the map \(f_*:\rt(D_1)\to \rt(D_2)\) is transducerable.
This map is given by the transducer \((D_1, \operatorname{id}_{D_1}, f)\).

\((1\Leftarrow 3):\) By definition, a transducerable map is the composite of two UDAFG isomorphisms.
\end{proof}

\section{Automorphism groups in the UDAFG category}\label{GNS section}
Part of the motivation for establishing the category UDAFG was so that we could capture the outer automorphism groups of \(V\)-like groups. In this section we explain how the category actually does this. 
We assume the reader is reasonably familiar with how the outer automorphism groups of \(V\)-like groups are described in the papers \cite{autgnr, autnv}.
We first recall the required definitions and results from \cite{autgnr}.
\begin{defn}[\(\mathcal{O}_n\) transducers, c.f. \cite{autgnr}]
If \(n\geq 2\) is an integer, then we define \(\mathcal{O}_n\) to be the group of 4-tuples \(T:=(X_n, Q_T, \pi_T, \lambda_T)\) (called \(\mathcal{O}_n\)-transducers), where:
\begin{enumerate}
    \item \(X_n:= \{0, 1, \ldots, n-1\}\) (called the alphabet of \(T\)).
    \item \(Q_T\) is a finite set (called the state set).
    \item \(\pi_T: (X_n \times Q_T) \to Q_T\) is a function (called the transition function).
    \item \(\lambda_T:(X_n \times Q_T) \to X_n^*\) is a function (called the output function). 
    Here \(X_n^*\) is the free monoid on the set \(X_n\).
    \item We extend the domain of \(\pi_T\) to \(X_n^* \times Q_T\) and the domain of \(\lambda_T\) to \((X_n^* \cup X_n^\omega) \times Q_T\) by the equalities
    \[\pi_T(aw, q)=\pi_T (\pi_T(a, q), w)\quad \text{ and }\quad \lambda_T(aw, q)=\lambda_T(a, q)\lambda_T(\pi_T(a, q), w)\]
    where \(a\in X_n\).
    \item There is some \(k>0\) such that for all \(w\in X_n^k\) and \(p,q\in Q_T\) we have \(\pi_T(w, p)=\pi_T(w, q)\) (we call the smallest such \(k\) the \textit{synchonizing length} of \(T\) and denote it by \(k_T\)).
    \item For all \(q\in Q_T\), there is \(w_q\in X_n^{k_T}\) such that \(\pi_T(w_q, q)=q\).
    \item The object \(T\) is minimal in the sense of \cite{GNS2000} (it has complete responce and no distinct states have the same output functions). 
    \item For each \(q\in Q_T\), the map \(\lambda_T(\cdot, q): X_n^\omega \to X_n^\omega\) is injective with clopen image.
    \item The transducer \(T\) is invertible.
\end{enumerate}
Here we identify two  \(\mathcal{O}_n\)-transducers if one can be obtained from the other by relabelling the states (as this group is countable, one can define this group more formally by choosing any infinite set and insisting on only allowing states from that set).
\end{defn}
\begin{theorem}[\(\Aut(G_{n, n-1})\), c.f. \cite{autgnr}]
If \(n\geq 2\) is an integer, then the groups \(\Out(G_{n, n-1})\) and \(\mathcal{O}_n\) are isomorphic.
\end{theorem}

The transducers above can be naturally converted into UDAF transducers as follows.
\begin{defn}[Converting transducers]\label{transducer conversion definition}
If \(T\) is an \(\mathcal{O}_n\) transducer then we define \(UDAF(T):= (D, f, g)\) as follows.
\begin{enumerate}
    \item \(D\) is the digraph with vertex set \(Q_T\) and edge set \(\makeset{(q, x, \pi_T(x, q))}{\(x\in X_n\)}\). 
    Here \((q, x, \pi_T(x,q))\) is an edge from \(q\) to \(\pi_T(x, q)\).
    \item \(R_n\) is the digraph with one vertex and edge set \(X_n\).
    \item \(f:D \to R_n\) is the digraph homomorphism with \((q, x, \pi_T(x, q))f_E= x\) for each \((q, x, \pi_T(x, q))\in E_D\).
    \item \(g:D \to \Wlk(R_n)\) is defined by \((q, x, \pi_T(x, q))g_E = \lambda_T(x, q)\) (we think of \(\lambda_T(x, q)\) as a walk through \(R_n\) in the natural fashion.)
\end{enumerate}
It follows from condition (6) of being an \(\mathcal{O}_n\) transducer that \(f\) is a folding.
Hence we have an action of the semigroup \(\mathcal{O}_n\) on the set \(\rt(R_n)\). 
Thus by condition (10) of being an \(\mathcal{O}_n\)-transducer, the group acts by permutations and so the map \(g\) above must also be an UDAF folding (but not necessarily a folding).
In particular \(UDAF(T)\) is always an UDAF transducer.
\end{defn}

\begin{defn}
If \(D\) is an UDAF digraph, then we denote the automorphism group of the object \((\Wlk(D), D)\) in the category UDAFG by \[\Aut_{UDAFG}(\Wlk(D), D).\]
\end{defn}
The following fact is well-known but the author is unaware if this fact appears anywhere is the literature. 
Thus we provide our own proof in the appendix. See Theorem~\ref{injectivity}.

\begin{remark}\label{appendix reference}
The action of \(\mathcal{O}_n\) on \(\rt(R_n)\) given in Definition~\ref{transducer conversion definition} is faithful. See Proposition~\ref{injectivity}.
\end{remark}

\begin{corollary}
If \(n\geq 2\) is an integer, then the group \(\Out(G_{n, n-1})\) embeds in the group \(\Aut_{UDAFG}(\rt(R_n), R_n)\).
\end{corollary}

\begin{theorem}\label{autgnr theorem}
If \(n\geq 2\) is an integer, then 
\[\Out(G_{n, n-1})\cong \Aut_{UDAFG}(\rt(R_n), R_n).\]
\end{theorem}
\begin{proof}
It suffices to show that if \((T, f, g)\) is an UDAF transducer inducing an UDAFG automorphism of \((\rt(R_n), R_n)\), then there is an \(\mathcal{O}_n\)-transducer \(T'\) which induces the same automorphism.
By Lemma~\ref{big lem}, we may assume without loss of generality that \(T\) is past-future digraph for \(R_n\) and \(f: T\to R_n\) is the standard UDAF folding. Here we view all walks in the digraph \(R_T\) as strings in the alphabet \(X_n\) (this includes the vertices and edges of \(T\)).

Let \(T':= (X_n, V_T, \pi_T, \lambda_T)\) where:
\begin{enumerate}
    \item \(V_T\) is the set of vertices of \(T\).
    \item \(\pi_T: X_n \times V_T  \to V_T\) is defined by assigning \(\pi_T(x, w)\) the unique vertex in \(V_T\) which is a suffix of the string \(wx\).
    \item \(\lambda_T: X_n \times V_T  \to V_T\) is defined by \(\lambda_T(x, v)= (vx)g_E\). 
\end{enumerate}

The transducer \(T'\) might not be minimal (so might not be an \(\mathcal{O}_n\)-transducer) however by construction it induces the same UDAFG automorphism of \((\rt(R_n), R_n)\) as \((T, f, g)\). It also satisfies all of the other conditions of being a \(\mathcal{O}_n\)-transducer (the 9th condition follows from Lemma~\ref{big lem}).
Thus if we minimise \(T'\) using the process described in \cite{autgnr}, then we obtain an \(\mathcal{O}_n\)-transducer which has the required action on \((\rt(R_n), R_n)\). 
\end{proof}

We next show that this extends to the groups \(nV\).
\begin{lemma}\label{wreath lemma}
Suppose that \(D\) is a strongly connected UDAF digraph and \(\Wlk(D)\neq \varnothing\).
If \(nD\) denotes the UDAF digraph consisting of \(n\) disjoint copies of \(D\), then 
\[\Aut_{UDAF}(\Wlk(nD), nD) \cong \Aut_{UDAF}(\Wlk(D), D) \wr S_n\]
where \(S_n\) denotes the symmetric group on \(n\) points with its usual action.
\end{lemma}
\begin{proof}
By Lemma~\ref{big lem} and Theorem~\ref{udaf isomorphisms theorem}, each UDAFG isomorphism of \((\Wlk(nD), nD)\) can be induced by an UDAF transducer \((T, f, g)\) where \(T\) is a past-future digraph of \(nD\) and \(f\) is the standard folding.
Note that a past-future digraph for \(nD\) is isomorphic to the disjoint union of \(n\) past-future digraphs for \(D\) (each of which is strongly connected).
The digraph homomorphism \(g:T\to \Wlk(nD)\) must therefore be the disjoint union \(n\) digraph homomorphisms from past-future digraphs of \(D\) to copies of \(D\). 
Each of copy of \(D\) is mapped to by precisely one of these digraph homomorphisms or otherwise \(g_*\) would not be surjective (\(\Wlk(D)\neq \varnothing\)).
It follows that \(g\) is the disjoint union of \(n\) UDAF folding from past-future digraphs of \(D\) to \(D\).
Moreover any such choice of \(g\) will be a folding from \(T\) to \(\Wlk(nD)\). 
Thus the wreath product follows.
\end{proof}

\begin{theorem}
If \(n\geq 2\) is an integer, then 
\[\Out(nV)\cong \Aut_{UDAFG}(\rt(nR_2), nR_2).\]
\end{theorem}
\begin{proof}
This is immediate from Theorem~\ref{autgnr theorem}, Lemma~\ref{wreath lemma} and the fact that \(\Out(nV)\cong \Out(V)\wr S_n\) (see \cite{autnv}).
\end{proof}

\part{Beams and multisets}\label{beams part}
This section is dedicated to the concept of beams and a particular equivalence relation on these beams. For background on indexed rays and beams see Section 7.5 of \cite{lind2021introduction}.
A beam is a collection of routes through an UDAF digraph that can be described as all the routes which begin with one ``infinite backwards route" from a specified finite collection.
Thus a beam is a union of finitely many sets which can each be described as the collection of all routes which begin with a single specified ``infinite backwards route" (such sets are called rays).
We show that UDAFG isomorphisms map beams to beams.
We also define a natural equivalence relation on the beams by considering two rays to be the same if their defining infinite backwards routes end at the same vertex (the end vertex cannot always be determined from the ray alone but it almost can).
If \(S\) is a set, then we consider a multiset of elements of \(S\) to be a function \(M:S\to \N\), where \(\N\) includes \(0\), and \((a)M\) is the number of copies of \(a\) in \(M\) (we will only consider finite multisets here so \(\N\) is sufficiently large).
The reader may find it helpful to think of \(M\) as a count or a weight system at times.
As \(\N^S\subseteq \Z^S\) we sometimes add/subtract multisets where this is useful.

\begin{defn}[Rays and beams]
If \(D\) is an UDAF digraph, \(w:\Z_{[-\infty, \infty]} \to D\) is such that \(w\restriction_{\Z_{[-\infty, 0]}}\) is irrational, and \(n\in \Z\) then we define
\[R(w, n):=\makeset{h\in \operatorname{Hom}(\Z_{[-\infty, \infty]}, D)}{\(h\restriction_{\Z_{[-\infty, n]}}= w\restriction_{\Z_{[-\infty, n]}}\)}\]
(note that this set is independent of \(w|_{\Z_{[n, \infty]}}\)).

In this case we say that \(n\) is a central index for \(R(w, n)\) and \((n)w_V\) is a central vertex for \(R(w, n)\).
The sets of this form are called \textit{rays}. 
Note that a ray always has a unique central index and vertex unless \(D\) has vertices of outdegree \(1\) (a ray is just a set of walks so it may be difficult to discern the central index \(n\) from the walks).

We then define \(R^u(w, n):= \makeset{\langle \sigma\rangle \circ h}{\(h\in R(w, n)\)}\), these are called \textit{unindexed rays}. 
Similarly in this case we say that \((n)w_V\) is a central vertex for \(R^u(w, n)\) (the notion of central index is no longer meaningful). We call a finite union of rays a beam, and a finite union of unindexed rays an unindexed beam. We define \(B^u(D)\) to be the set of unindexed beams of \(D\).
\end{defn}

The insistence that \(p\) be irrational in the previous definition is so that we may easily ``split" a ray into subrays.
For example, if \(R_2\) is the 2-leaf rose with edges \(\{a, b\}\) and \[w=\ldots aba \]
is some infinite backwards walk, then we would like the rays corresponding to \(w\) to be the disjoint union of the rays corresponding to 
\[wa=\ldots abaa \quad \text{ and }\quad wb=\ldots abab.\]
In particular we would not like there to be a single route which has both of \(wa\) and  \(wb\) as infinite tails. 
The irrationality of \(w\) guarantees this. Hence we have the following remark.
\begin{remark}
If \(D\) is an UDAF digraph and \(B\) is an unindexed beam of \(D\), then \(B\) can be expressed as a \textbf{disjoint} union of unindexed rays.
\end{remark}
\begin{remark}
It follows from Lemma~\ref{big groupoid} that UDAFG isomorphisms map unindexed beams to unindexed beams (as the elementary UDAF foldings and their inverses do).
\end{remark}
\begin{defn}
If \(D\) is an UDAF digraph, then we define the \textit{core} \(D^\circ\) of \(D\) to be the subdigraph of \(D\) consisting of those edges and vertices with are involved in bi-infinite walks through \(D\).
\end{defn}
\begin{defn}[Counts and equivalence relations]\label{counts and equivalence relations defn}
Let \(B\) be an unindexed beam of an UDAF digraph \(D\), then we can write \(B\) as a disjoint union of finitely many unindexed rays \(R_0, R_1 \ldots R_{k-1}\). 
If \(M:V_{D^\circ}\to \N\) is a multiset of vertices of \(V_{D^\circ}\), and there is an indexed collection of vertices \((v_i)_{i<k}\) with each \(v_i\) a central vertex of \(R_i\), such that for all \(v\in V_{D^\circ}\) we have
\[(v)M=\left|\makeset{i<k}{\(v_i=v\)}\right|\]
then we say that \textit{\(M\) is a vertex count for \(B\)}. 
A typical beam will have infinitely many vertex counts, as there are typically many ways of writing a beam as a union of rays, and a ray may have multiple choices of central vertex.

We define a relation \( \sim_{D}\) on \(\N^{V_{D^\circ}}\) to be the equivalence relation generated by the pairs 
\[\makeset{(M_1, M_2)}{there is an unindexed beam \(B\) such that \(M_1\)\\ and
\(M_2\) are both vertex counters for \(B\)}.\]
We then define \(\N^{V_{D^\circ}}/ \sim_{D}:= \makeset{[M]_{\sim_D}}{\(M\in V_{D^\circ}\)}\).
We define a map \(C_D:B^u(D)\to \N^{V_{D^\circ}}/ \sim_{D}\) by 
\[(B)C_D=[M]_{ \sim_{D}}\]
where \(M\) is any vertex counter for \(B\) (by the definition of \(\sim_{D}\) this map is well defined).
We define \(\sim_{D, B}\) to be the kernel of \(C_D\), that is the equivalence relation
\[\makeset{(B_1, B_2)\in B^u(D)}{\((B_1)C_D=(B_2)C_D\)}.\]
\end{defn}
\begin{defn}[Refinements and the refining monoid]
If \(D\) is an UDAF digraph and \(A, B\in \N^{V_{D^\circ}}\), then we say that \(B\) is an \textit{elementary refinement} of \(A\) if it can be obtained from \(A\) via the following process: 

Choose a vertex \(v\) of \(V_{D^\circ}\) with \((v)A>0\) and let \(P\) be a collection finite walks in \(D^\circ\) starting at \(v\) such that every infinite forwards walk from \(v\) begins with precisely one element of \(P\).
Let \(S\) be the multiset of vertices which are the end vertices of elements of \(P\).
Let \(R:V_{D^\circ}\to \Z\) be the same as \(S\) except with \((v)R= (v)S-1\). Finally let \(B= A+R\).

We call such an \(R\) an \textit{elementary refiner}, and the monoid generated by such maps \(R:V_{D^\circ} \to \Z\), the \textit{refining monoid} \(\operatorname{Refine}(D)\). Observe that \(\operatorname{Refine}(D)\) is a submonoid of \(\Z^{V_{D^\circ}}\).

We say that \(B\) is an \textit{refinement} of \(A\) if there is a sequence of multisets \(A=M_0, M_1, \ldots M_k = B\) such that each \(M_{i+1}\) is obtained from \(M_i\) by an \textit{elementary refinement}.
\end{defn}
\begin{defn}
If \(D\) is an UDAF digraph, \(M\in \N^{V_{D^\circ}}\), \(v\in V_{D^\circ}\), \((v)M>0\), and \(n\in \N\), then we define the \textit{standard} \((v, n)\)-refinement of \(M\) to be the elementary refinement using the vertex \(v\) and the anti-chain of all walks (in \(D^\circ\)) starting at \(v\) of length \(n\). 
More generally we call these \textit{standard} refinements.
Note that a standard \((v, 0)\)-refinement of \(M\) is always \(M\).
\end{defn}
\begin{remark}\label{easy refinements}
Let \(D\) be an UDAF digraph and \(M_1\in \N^{V_{D^\circ}}\).

Any elementary refinement \(M_2\) of a multiset \(M_1\) can be further refined to obtain a standard refinement \(M_3\) of \(M\) using the same vertex used to refine \(M_1\) to \(M_2\).

Moreover all refinements \(M_2\) of \(M_1\) can be further refined to obtain a refinement \(M_3\) of \(M_1\) which could have been obtained by performing a standard refinement from each vertex in \(M_1\).

\end{remark}
\begin{lemma}\label{refinement gives refined partition}
Let \(D\) be an UDAF digraph. If \(M_1, M_2 \in \N^{V_{D^\circ}}\) and one is a refinement of the other, then there is \(B\in B^u(D)\) such that both \(M_1\) and \(M_2\) are vertex counters for \(B\). In particular \(M_1\sim_{D} M_2\).
\end{lemma}
\begin{proof}
We can assume without loss of generality that \(M_2\) is an elementary refinement of \(M_1\).
Let \(R_0, R_1, \ldots, R_{n-1}\) be disjoint unindexed rays of \(D\) such that there are central vertices \((v_i)_{i<n}\) of \(R_i\) with \( (v)M_1=\left|\makeset{i<k}{\(v_i=v\)}\right|\)
for all \(v\in V_{D^\circ}\).

Assume without loss of generality that \(M_2\) is a refinement of \(M_1\) using the vertex \(v_0\) and a collection of walks \(P\).
Let \(w:\Z_{[-\infty, \infty]}\to D\) be a bi-infinite walk with \((0)w_V= v_0\) and such that \(R_0= R^u(w, 0)\).
 From the definition of an unindexed ray, it follows that \(w\restriction_{\Z_{[-\infty, 0]}}\) is irrational.
 For each \(p\in P\), let \(w_p:\Z_{[-\infty, \infty]} \to D\) be such that \(w_p\restriction_{\Z_{[-\infty, 0]}} = w\restriction_{\Z_{[-\infty, 0]}}\) and  \(w_p\restriction_{\dom(p)} = p\).

 It follows from the definition of \(P\) that the sets \(R(w_p, \norm{p})\) for \(p\in P\) are a partition of \(R(w, 0)\) and each of these has the end vertex of the corresponding element of \(p\) as a central vertex.
 
It follows that the unindexed rays \(R_1, \ldots, R_{n-1}\) together with \(R^u(w_p, \norm{p})\) for \(p\in P\) are a partition of \(\cup_{i<n} R_i\).
In particular \(M_1\) and \(M_2\) are both vertex counters for this unindexed beam.
\end{proof}

Recall that partition \(P_2\) \textit{refines} another partition \(P_1\) if each element of \(P_1\) is a union of elements of \(P_2\).
\begin{lemma}\label{rays_refined}
Suppose that \(D\) is an UDAF digraph, \(B\) is an unindexed beam of \(D\), and \(P_1, P_2\) are each partitions of \(B\) into finitely many unindexed rays.
There is a partition \(P_3\) of \(B\) into unindexed rays such that \(P_3\) is a refinement of both \(P_1\) and \(P_2\).
\end{lemma}
\begin{proof}
Let \(\{R_{i, 1}: i\in \{0, 1, \ldots k-1\}\}\) and \(\{R_{i, 2}: i\in \{0, 1, \ldots l-1\}\}\) be partitions of \(B\) into unindexed rays. It follows that the set \(\{R_{i, 1}\cap R_{j, 2}:i<k, j<l\}\backslash\{\varnothing\}\) is a finite partition of \(B\) refining both of our previous ones. 
It thus suffices to show that a non-empty intersection of two unindexed rays is an unindexed ray.

Let \(R_1, R_2\) be unindexed rays with \(R_1\cap R_2\neq \varnothing\).
Let \(w_1, w_2:\Z_{[-\infty, \infty]}\to D\) be bi-infinite walks with \(R_1=R^u(w_1, 0)\) and \(R_2= (w_2, 0)\).
As \(R_1\) and \(R_2\) are symmetric and \(R_1\cap R_2\neq \varnothing\), we may assume without loss of generality that there is \(n\in \N\) with \(w_1 \restriction_{\Z_{[-\infty, 0]}}= (\sigma^{-n} w_2)\restriction_{\Z_{[-\infty, 0]}}\). 
It follows that \(R_2\subseteq R_1\). Thus \(R_1\cap R_2=R_2\) and is thus an unindexed ray.
\end{proof}
\begin{theorem}\label{congruence def 2}
If \(D\) is an UDAF digraph and \(M_1, M_2\in \N^{V_{D^\circ}}\) then \(M_1  \sim_{D} M_2\) if and only if they have a common refinement.
\end{theorem}
\begin{proof}
\((\Leftarrow):\) Suppose that \(M_1, M_2\) have a common refinement \(M_3\). Then \(M_1 \sim_{D} M_3  \sim_{D} M_2\) (by Lemma~\ref{refinement gives refined partition}).

\((\Rightarrow):\) Suppose that \(M_1\) and \(M_2\) are both vertex counters for a common beam \(B\), corresponding to partitions \(\{R_{1, i}:i<k\}\) and \(\{R_{2, i}:i<l\}\) of \(B\) into rays. By Lemma~\ref{rays_refined} these two partitions have a common refined partition \(RP\) of \(B\) into rays. 

Let \(i<k\) be arbitrary, and let \(v_i\) be a central vertex for \(R_{1, i}\).
Let \(R\in RP\) be arbitrary such that \(R\subseteq R_{1, i}\).
The irrational tail shared by the elements of \(R\) is an extension of the irrational tail shared by the elements of \(R_{1, i}\) followed by some finite walk \(p_R\) from the vertex \(v_i\) to a central vertex for \(R\). 
As \(RP\) is a partition, for a fixed \(i\), the set \(P_i\) of such finite walks \(p_R\) has the property that every infinite forwards walk from \(v_i\) starts with precisely one of these finite walks \(p_{R}\).
Thus by performing elementary expansions from each of the vertices \(v_i\) using the set of walks \(P_i\), it follows that \(M_3\) is a refinement of \(M_1\).
Similarly, \(M_3\) is a refinement of \(M_2\) and so \(M_1\) and \(M_2\) have a common refinement.

It therefor suffices to show that having a common refinement is an equivalence relation. Having a common refinement is clearly reflexive and symmetric so we show that it is transitive.

First note that the empty (constant \(0\)) multiset is not related to any other so we can safely ignore it. Suppose that \(M_1, M_2, M_3, M_{1, 2}, M_{2, 3}\in \N^{V_{D^\circ}}\), \(M_{1, 2}\) is a common refinement of \(M_1, M_2\), and \(M_{2, 3}\) is a common refinement of \(M_2, M_3\).

By Remark~\ref{easy refinements}, there is \(R_1:M_2 \to \N\) (here we mean to assign each vertex in the multiset \(M_2\) a natural number, possibly assigning different numbers to different copies of a single vertex) such that the refinement \(M_{1,2}'\) of \(M_2\) obtained by performing a standard \((v, (v)R_1)\)-refinement for each \(v\in M_{2}\) is a refinement \(M_{1, 2}\).
Similarly, there is \(R_2:M_2 \to \N\) such that the refinement \(M_{2, 3}'\) of \(M_2\) obtained by performing a standard \((v, (v)R_2)\)-refinement for each \(v\in M_2\) is a refinement \(M_{2, 3}\).
Let \(R_3\) be the coordinate-wise max of \(R_1\) and \(R_2\). 
It follows that the refinement of \(M_2\) obtained by performing a standard \((v, (v)R_3)\)-refinement for each \(v\in M_2\) is a refinement of both \(M_{1, 2}\) and \(M_{2, 3}\) and is thus a common refinement of both \(M_1\) and \(M_3\).
\end{proof}
\begin{corollary}\label{counting_refinement}
If \(D\) is an UDAF digraph, \(M\in \N^{D^\circ}\), \(B\in B^u(D)\) and \((B)C_D=[M]_{ \sim_{D}}\), then \(M\) has a refinement which is a vertex counter for \(B\).
\end{corollary}
\begin{proof}
Lemma~\ref{refinement gives refined partition} together with Theorem~\ref{congruence def 2}.
\end{proof}
\begin{defn}[Conversion matrices]
Suppose that \(f:D_1\to \Wlk(D_2)\) is an UDAF folding and \(M_f:V_{D_1^\circ}\times V_{D_2^\circ} \to \N\)
is a matrix (with rows corresponding to vertices of \(D_1^\circ\) and columns corresponding to vertices of \(D_2^\circ\)).
We say that \(M_f\) is a \textit{conversion matrix} for \(f\) if for all \(B\in B^u(D_1)\) and vertex counters \(M_B\) of \(B\), we have \(M_BM_f\) (viewing multisets as row vectors) is a vertex counter for \(B_1f_*\). 
The following theorem guarantees that these always exist.
\end{defn}
\begin{theorem}(Conversion matrices work)\label{beam type function}
Let \(D_1, D_2\) be UDAF digraphs and let \(f\in \Hom(D_1, \Wlk(D_2))\) by an UDAF folding. There is a conversion matrix for \(f\).
Moreover the map 
\[[B]_{\sim_{D_1, B}} \to [(B)f_*]_{\sim_{D_2, B}}\]
 from \(B^u(D_1)/\sim_{D_1, B}\) to \(B^u(D_2)/\sim_{D_2, B}\) is well defined (recall Definition~\ref{counts and equivalence relations defn}).
\end{theorem}
\begin{proof}
 Suppose that \((B_1)C_{D_1}=(B_2)C_{D_1}=[M]_{\sim_{D_1, M}}\), for some fixed multiset \(M\).
 By Corollary~\ref{counting_refinement}, we may assume without loss of generality that \(M\) is a vertex counter for both \(B_1\) and \(B_2\). For each \(v\in V_{D_1^\circ}\) we know (by Lemma~\ref{big lem}) that the set \(\fim_f(v)\) is clopen. 
 That is, we can choose a finite set of finite walks \(P_v\) in \(D_2\) starting at \((v)f_V\) and ending at various vertices in \(D_2\), with the property that the set of infinite forwards walks in \(D_1\) starting at \(v\) is mapped by \(f\) to the set of infinite forwards walks in \(D_2\) beginning with elements of \(P_v\) (the choice of \(P_v\) is not unique).

We define a matrix \(M_f:V_{D_1^\circ}\times V_{D_2^\circ}\to \N\) (so the vertices in \(V_{D_1^\circ}\) correspond to rows and the vertices in \(V_{D_2^\circ}\) correspond to columns) by
\[(v, v')M_f:= |\{p\in P_v:p\text{ ends at }v'\}|.\]
Note that this matrix is a conversion matrix for \(f\).
It follows that
\[(B_1f_*)C_{D_2}= \left[MM_f\right]_{\sim_{D_2, V}}=(B_2f_*)C_{D_2}\]
and so \(B_1f_* \sim_{D_2, B} B_2f_*\) as required.
\end{proof}
As one might hope, the induced maps between equivalence classes of unindexed beams are always bijections. 
This is the topic of the next lemma and theorem.
\begin{lemma}\label{past-future beam bijection lemma}
If \(D\) is an UDAF digraph and \(n\in \N\), then the natural \(UDAF\) folding from \(L:=\PF(D, -n, 0)\) to \(D\) induces a bijection between the sets \(B^u(L)/\sim_{L, B}\) and \(B^u(D)/\sim_{D, B}\).
\end{lemma}
\begin{proof}
We assume without loss of generality that \(D=D^\circ\). By Theorem \ref{beam type function} we know that \(f_*\) induces a function \(f_*'\) from \(B^u(L)/\sim_{L, B}\) to \(B^u(D)/\sim_{D, B}\). 
As \(f_*\) is surjective, it follows that \(f_*'\) is also surjective. 
It remains to show that \(f_*'\) is injective.

Suppose that \([B_1]_{\sim_{L, B}}f_*'=[B_2]_{\sim_{L, B}}f_*'\). 
It follows that \(((B_1)f_*)C_D=((B_2)f_*)C_D\).

Let \(M'\) be a common vertex counter for \(B_1f_*\) and \(B_2f_*\). Let \(\{R_{0},R_{1}, \ldots R_{k-1}\}\) be a partition of \(B_1f_*\) into unindexed rays, and \((v_i')_{i<k}\) be such that \(v_i'\) is a central vertex for \(R_i\) and for all \(v'\in V_{D^\circ}\) we have
\[(v')M'=|\makeset{i<k}{\(v_i'=v\)}|.\]
By the choice of \(f\), we know that each \(R_if_*^{-1}\) is a single unindexed ray (you always have \(n\) edges of past information).
Moreover, for each \(i\), we can choose a central vertex \(v_{i}\) of \((R_i)f^{-1}\) with \((v_i)f_V=v_i'\). 
We define \(M_1\) to be the multiset of the vertices \(v_{i}\), observe that \(M_1\) is a vertex counter for \(B_1\).

For each vertex \(v\) of \(L\) let \(M_v\) denote the multiset of vertices at the end of the length \(|n|+1\) walks in \(L\) starting at \(v\). 
It follows that
\[\widetilde{M_1}:= \sum_{v\in V_{L}}((v)M_1) \cdot M_v\]

is a refinement of \(M_1\) and hence a vertex counter for \(B_1\).
By the choice of \(L\), the multisets \(M_{v_1}\) and \(M_{v_2}\) agree whenever \((v_1)f=(v_2)f\).
Thus if \(v'\in V_D\) is arbitrary and \(v\in (v')f_V^{-1}\), then we may define \(M_{v'}:=M_{v}\) (these are multisets of vertices of \(L\)).
Moreover, by the construction of \(M_1\), we have for all \(v'\in V_{D}\) that
\[\sum_{v\in (v')f_V^{-1}}(v)M_1= (v')M'.\]
We obtain
\begin{align*}
\widetilde{M_1}&= \sum_{v\in V_{L}}((v)M_1)\cdot  M_v=\sum_{v'\in V_D}\left(\sum_{v\in (v')f_V^{-1}}((v)M_1) \cdot M_{v'}\right)\\
  &=\sum_{v'\in V_D}\left(\sum_{v\in (v')f_V^{-1}}((v)M_1) \right)\cdot M_{v'}=\sum_{v'\in V_D}((v')M')\cdot M_{v'}.
\end{align*}
It follows that if we construct \(\widetilde{M_2}\) analogously to be a vertex counter for \(B_2\) we obtain the same multiset. So \(B_1\) and \(B_2\) have a common vertex counter as required.
\end{proof}
\begin{theorem}\label{beam bijections}
If \(D_1, D_2\) are UDAF digraphs and \(\phi:\rt(D_1)\to \rt(D_2)\) is an UDAFG isomorphism, then the map \(\phi':B^u(D_1)/\sim_{D_1, B}\to B^u(D_2)/\sim_{D_2, B}\) defined by
\[([B]_{\sim_{D_1, B}})\phi' = [(B)\phi]_{\sim_{D_2, B}}\]
is a well-defined bijection.
\end{theorem}
\begin{proof}
By the definition of UDAFG, it suffices to consider the case that \(\phi=f_*\) for some UDAF folding \(f:D_1\to \Wlk(D_2)\). 
By Lemma~\ref{big lem}, there is a past-future digraph \(P\) of \(D_1\) and an UDAF folding \(g\in \Hom(P, \Wlk(D_1))\)) such that \(g_*f_*\) is the standard UDAF folding from \(P\) to \(D_2\). 

Note that for all \((m, n)\in \Z_{\leq 0} \times \Z_{\geq 0}\), there is a natural digraph isomorphism \(\psi:PF(D_2, m,n)\to PF(D_2, m-n,0)\) which preserves the UDAFG isomorphism induced by their standard UDAF foldings.
It thus follows from Lemma~\ref{past-future beam bijection lemma} that the function \((g_*f_*)'=g_*'f_*'\) is a bijection.

Thus \(f_*'\) is a surjection. 
It remains to show that \(f'\) is an injection. 
The map \(f_*(g_*f_*)^{-1}g_*=f_*(f_*)^{-1}(g_*)^{-1}g_*\) is the identity map on the set \(\rt(D_1)\). 
Thus the map \(f_*'((g_*f_*)')^{-1}g_*'\) is also the identity map on the set \(B^u(D_1)/\sim_{D_1, B}\). 
It follows that \(f_*'\) is also injective as required.
\end{proof}
\begin{defn}
If \(D\) is an UDAF digraph, then we define \(\N^{V_{D^\circ}}\) to be the \textit{multiset monoid} of \(D\). It is notably generated (as a monoid) by the elements \(\makeset{\chi_{D, v}}{\(v\in V_D^\circ\)}\) where \(\chi_{D, v}:V_D^\circ \to \N\) is defined by
\[(w)\chi_{D, v}=\left\{\begin{array}{lr}
1     &  w = v \\
0     &  w \neq v
\end{array}\right\}.\]
We then define the \textit{multiset semigroup} \(MS(D)\) of \(D\) to be the semigroup generated by \(\makeset{\chi_{D, v}}{\(v\in V_D^\circ\)}\). This is the same as \(\N^{V_{D^\circ}}\) without its identity.
\end{defn}

The reader may wonder why we consider \(MS(D)\) at all, instead of exclusively using \(\N^{V_{D^\circ}}\). 
This is because the relation \(\sim_{D}\) is a congruence on \(\N^{V_{D^\circ}}\), and we will show later that in many cases the semigroup \(MS(D)/\sim_{D}\) is actually a group. 

\begin{lemma}
If \(D\) is an UDAF digraph, then the relation \( \sim_{D}\) is a congruence on the multiset monoid (and hence multiset semigroup) of \(D\).
\end{lemma}
\begin{proof}
Suppose that \(M_1 \sim_{D} M_2\) and \(M_3 \sim_{D} M_4\).
It suffices to show that \(M_1+M_3 \sim_D M_2 +M_4\).

By Theorem~\ref{congruence def 2}, we can find a common refinement \(M_{1, 2}\) of both \(M_1\) and \(M_2\). Similarly we can find a common refinement \(M_{3, 4}\) of both \(M_3\) and \(M_4\).

The multiset \(M_1+M_3\) is greater than \(M_1\), thus any refinements we can apply to \(M_1\) can be applied to \(M_1+M_3\).
It follows that \(M_{1,2}+M_3\) is a refinement of \(M_1+M_3\). By the same argument applied to \(M_3\) it follows that \(M_{1,2}+M_{3, 4}\) is a refinement of \(M_{1,2}+M_3\).

Similarly \(M_{1, 2} + M_{3, 4}\) is a refinement of \(M_2+M_4\), so \(M_1+M_3\) and \(M_2 + M_4\) have a common refinement. Thus \(M_1+M_3\sim_D M_2 + M_4\) (by Theorem~\ref{congruence def 2}).
\end{proof}
\section{UDAF dimension functor}
In this section we use the results of the previous section to define and explore the UDAF analogue to dimension modules.
\begin{defn}
We now define a functor \(\operatorname{Dim}\) from the UDAFG category to the category of monoids and isomorphisms.
If \((\rt(D),D)\) is a object of UDAFG, \(\phi\) is an UDAFG isomorphism, and \(B\in B^u(D)\) then we define
\begin{align*}
    (\rt(D),D)\operatorname{Dim}&=\N^{V_D^\circ}/ \sim_{D},\\
([(B)C_D]_{ \sim_{D}})(\phi)\operatorname{Dim}&=[(B)\phi C_D]_{\sim_D}
\end{align*}
(by Theorem~\ref{beam bijections} \((\phi)\operatorname{Dim}\) is a well-defined bijection and by the first part of Theorem~\ref{beam type function} it is a semigroup homomorphism).

Due to its connection to dimension modules (Definition~\ref{dimension module defn}) we call the image of this functor, from the UDAFG category, the UDAF dimension category. 
Moreover if \(D\) is an UDAF digraph, then we call \(\N^{V_D^\circ}/\sim_D\) and \(MS(D)/\sim_D\) the \textit{UDAF dimension monoid} and \textit{UDAF dimension semigroup} of \(D\) respectively.
\end{defn}
\begin{defn}
We say that two UDAF digraphs \(D_1, D_2\) are strong UDAF equivalent if \((\rt(D_1),D_1)\) and \((\rt(D_2),D_2)\) are isomorphic in the UDAFG category, and we say that they are weak UDAF equivalent if the monoids \((D_1)\operatorname{Dim}\) and \((D_2)\operatorname{Dim}\) are isomorphic.
\end{defn}
Note that if \(D_1\) and \(D_2\) are strong UDAF equivalent, then they are also weak UDAF equivalent.

\begin{defn}
If \(D\) is an UDAF digraph, then we define its canonical adjacency matrix \(\adj_D:V_D^2 \to \N\), canonical identity matrix \(I_D:V_D^2 \to \N\) and canonical relator matrix \(\rel_D:V_D^2 \to \Z\) by
\begin{align*}
    (v, w)\adj_D&:= \left|\makeset{e\in E_D}{\(((e)s_D, (e)t_D)=(v,w)\)}\right|,\quad (v, w)I_D&:= \left\{\begin{array}{cc}
        1 & \text{if }v= w\\
        0 & \text{if }v\neq w
    \end{array}\right\},
\end{align*}
and  \(\rel_D:= \adj_D-I_D\).
Similarly if \(S\) is any set, \(\phi:V_D\to S\) is a bijection, and \(A:S^2\to \N\), \(I:S^2\to \N\) and \(R:S^2\to \Z\) are defined by

\begin{align*}
    (a, b)A&:= \left|\makeset{e\in E_D}{\(((e)s_D, (e)t_D)=((a)\phi^{-1},(b)\phi^{-1})\)}\right|,\quad (a, b)I&:=\left\{\begin{array}{cc}
        1 & \text{if }a= b\\
        0 & \text{if }a\neq b
    \end{array}\right\}
\end{align*}
and  \(R:= A-I\),
then we say that \(A,I, R\) are respectively an adjacency matrix, an identity matrix, and a relator matrix for \(D\).
\end{defn}
The name ``relator matrix" was chosen because of Lemma~\ref{congruence is nice} in which the rows of this matrix are relators in an abelian group presentation.
Note that two digraphs have a common relator matrix if and only if they have a common adjacency matrix. Hence if they have a common relator matrix then they are isomorphic.

\begin{theorem}\label{refiner rows}
Let \(D\) be an UDAF digraph with \(D=D^\circ\). 
The UDAF dimension monoid of \(D\) is presented as a commutative monoid by:
\[\left\langle V_{D^\circ} \left|
\text{ for all }v\in V_{D^\circ}\text{ impose the relation that }v=\sum_{w\in V_{D^\circ}}  \left({(v, w)\operatorname{Adj}_{D^\circ}}\cdot  w\right) \right.
\right\rangle_{CM}.\]
Moreover the refining monoid \(\operatorname{Refine}(D)\) is generated by the rows of \(\operatorname{Rel}_{D^\circ}\).
\end{theorem}
\begin{proof}
A row \(v\) of the canonical relator matrix for \(D\) corresponds precisely to what is added to a multiset when a standard \((v, 1)\)-refinement is performed to it.
Thus the congruence on \(\N^{V_{D^\circ}}\) defined in the above presentation is contained in \(\sim_D\).

Moreover if \(M_1, M_2\in \N^{V_{D^\circ}}\) and \(M_2\) is a refinement of \(M_1\), then \(M_2\) can be obtained from \(M_1\) via a sequence of standard \((v, 1)\)-refinements (for various choices of \(v\)). 
In particular the refining monoid \(\operatorname{Refine}(D)\) is generated by the rows of \(\operatorname{Rel}_{D^\circ}\) and the congruences agree by Theorem~\ref{congruence def 2}.
\end{proof}
\begin{corollary}
Suppose that a digraph \(D\) can be expressed as a disjoint union of UDAF digraphs \(D_0, D_1, \ldots , D_{k-1}\). 
Then \(D\) is an UDAF digraph and the UDAF dimension monoid of \(D\) is isomorphic to the free product
of the UDAF dimension monoids of \(D_0, D_1, \ldots, D_{k-1}\) as commutative monoids (which is notably the same as their direct product).
\end{corollary}
\begin{lemma}\label{strongly connected congruence}
If \(D\) is a strongly connected UDAF digraph, and  \(M_1, M_2\in MS(D)\) then \(M_1 \sim_{D} M_2 \iff( M_1+\operatorname{Refine}(D)) \cap (M_2 + \operatorname{Refine}(D)) \neq \varnothing\).
\end{lemma}
\begin{proof}
We show the reverse implication, the forwards implication is immediate from Theorem~\ref{congruence def 2}. 
Suppose that \(M_1, M_2\in MS(D)\) and \(M_3=M_1+R_1=M_2+R_2\) for some \(R_1, R_2\in \operatorname{Refine}(D)\).

Let \(n := \max\left(\makeset{|(v)R_1|+|(v)R_2|+|(v)M_1| + |(v)M_2|}{\(v\in V_D\)}\right)\). 
As \(D\) is a strongly connected UDAF digraph, we can find \(R_3, R_4\in \operatorname{Refine}(D)\) such that \(M_1+R_3\) and \(M_2+R_4\) are refinements of \(M_1, M_2\) respectively and each of \(R_3, R_4\) has minimum entry at least \(n\). It follows from the choice of \(n\) that
\[M_1+R_3+R_1 + R_4= M_2+R_4 + R_2 + R_3,\]
is a common refinement of \(M_1\) and \(M_2\).
In particular \(M_1 \sim_{D} M_2\) as required.
\end{proof}
\begin{lemma}\label{congruence is nice}
If \(D\) is a strongly connected UDAF digraph,  \(M_1, M_2\in MS(D)\), and \( \langle \operatorname{Refine}(D)\rangle\) is the additive group generated by \(\operatorname{Refine}(D)\), then
 \[M_1 \sim_{D} M_2 \iff M_2\in M_1+\operatorname{Refine}(D)-\operatorname{Refine}(D) \iff M_2\in M_1 + \langle \operatorname{Refine}(D)\rangle.\]
\end{lemma}
\begin{proof}
This is immediate from Lemma~\ref{strongly connected congruence}.
\end{proof}
\begin{theorem}\label{UDAF dimension group}
If \(D\) is a strongly connected UDAF digraph, then the UDAF dimension semigroup of \(D\) is an abelian group, with abelian group presentation
\[\left\langle V_{D^\circ} \left|
\text{ for all }v\in V_{D^\circ}\text{ impose the relation that }0=-v+\sum_{w\in V_{D^\circ}}  \left({(v, w)\operatorname{Adj}_{D^\circ}}\cdot  w\right) \right.
\right\rangle_A\]
or the empty semigroup when \(D\) has no edges.
\end{theorem}
\begin{proof}
We assume that \(D\) has edges, hence \(D^\circ = D\). Let
\(G := \Z^{V_{D^\circ}}/\langle \operatorname{Refine}(D)\rangle.\)
Let \(f:MS(D)\to G\) be the semigroup homomorphism defined by
\[(M)f = M+\langle \operatorname{Refine}(D)\rangle.\]
By Lemma~\ref{congruence is nice} and Theorem~\ref{refiner rows}, the semigroup homomorphism \(f\) has kernel \( \sim_{D}\).
Thus by the First Isomorphism Theorem for semigroups \(MS(D)/ \sim_{D}\cong \im(f)\).
It therefore suffices to show that \(\im(f)\) is inverse closed (note that \(\im(f)\) is generated as a semigroup by \(\makeset{(\chi_{D,v})f}{\(v\in V_D^\circ\)}
=\makeset{\chi_{D,v} + \langle \operatorname{Refine}(D)\rangle}{\(v\in V_D^\circ\)}\)).

As \(D\) is a strongly connected UDAF digraph, it follows that for all \(n\in \N\) we can find \(R_n\in \operatorname{Refine}(D)\) with \((v)R_n>n\) for all \(v\). Let \(M+ \langle \operatorname{Refine}(D)\rangle\in \im(f)\), with inverse \(-M+ \langle \operatorname{Refine}(D)\rangle\). Let \(n\in \N\) be such that \(R_n-M\) is positive in every coordinate. Then 
\[-M +\langle \operatorname{Refine}(D)\rangle=R_n-M+ \langle \operatorname{Refine}(D)\rangle=(R_n-M)f\in \im(f). \]
\end{proof}
\begin{corollary}
If \(D_1\) and \(D_2\) are strongly connected UDAF digraphs with edges and \(|D_1|=|D_2|\), then \(D_1\) and \(D_2\) are weak UDAF equivalent if and only if \(\operatorname{Rel}_{D_1}\) and \(\operatorname{Rel}_{D_2}\) have the same Smith normal form (see \cite{Smithnormalform} for background on Smith normal form).
\end{corollary}
\begin{corollary}
For all \(n, m\geq 2\), the n-leaf rose \(R_n\) and m-leaf rose \(R_m\) are weak UDAF equivalent if and only if \(n=m\).
\end{corollary}
\section{The matrix viewpoint}
Up to this point we have had a strong focus on digraphs with matrices being used where particularly useful.
However a digraph as we define it is essentially equivalent to its adjacency matrix.
Hence all of the statements in this document can theoretically be converted into the language of matrices.
In this section we make some observations concerning UDAF equivalence of matrices. 
These will be useful in practice when checking UDAF equivalence.
\begin{lemma}\label{elmat1}
If \(D_1\) is an UDAF digraph obtained from another UDAF digraph \(D_2\), via an elementary UDAF folding of the first type, then any matrix obtained from a relator matrix of \(D_2\) by inserting a cross centered at \((i, i)\) (for any \(i\)) with \(-1\) on position \((i, i)\) and zeros elsewhere, is a relator matrix for \(D_1\) (Recall Definition~\ref{elementary UDAF foldings defn}).
\end{lemma}
\begin{proof}
This is because adding an isolated vertex adds a cross of zeros centered on the diagonal to the adjacency matrix (the center of the cross is based on how the new vertex is ordered in comparison to the others).
\end{proof}
\begin{example}
Here are some examples of the change of relator matrix as described in Lemma \ref{elmat1}.
\[\left[\begin{smallmatrix}
    a & b\\
    c & d
\end{smallmatrix}\right] \rightarrow \left[\begin{smallmatrix}
    a & 0 & b\\
    0 & -1 & 0\\
    c & 0 & d
\end{smallmatrix}\right]\quad\quad\quad\quad \left[\begin{smallmatrix}
    a & b\\
    c & d
\end{smallmatrix}\right] \rightarrow \left[\begin{smallmatrix}
    a & b & 0\\
    c & d & 0\\
    0 & 0 & -1
\end{smallmatrix}\right]\]
\[\left[\begin{smallmatrix}
    a & b & c\\
    d & e & f\\
    g & h & i
\end{smallmatrix}\right] \rightarrow \left[\begin{smallmatrix}
    a & b &  0&c\\
    d & e &  0&f\\
    0 & 0 & -1 & 0\\
    g & h &  0&i
\end{smallmatrix}\right]\]
\end{example}
\begin{lemma}\label{elmat2}
If \(D_1\) is an UDAF digraph obtained from another UDAF digraph \(D_2\), via an elementary UDAF folding of the second type via an edge \(e\in E_{D_1}\), then the matrix obtained from the canonical relator matrix of \(D_1\) by adding the \((e)t_{D_1}\) row to the \((e)s_{D_1}\) row is the canonical relator matrix for \(D_2\).
\end{lemma}
\begin{proof}
The effect of elementary UDAF foldings of the second type is to remove the edge \(e\) (subtracting 1 from the \(((e)s_{D_1}, (e)t_{D_1})\) entry of the canonical adjacency matrix), and add a new edge starting at \((e)s_{D_1}\) for each edge starting at \((e)t_{D_1}\) (adding the \((e)t_{D_1}\) row to the \((e)s_{D_1}\) row of the canonical adjacency matrix).
Thus in the relator matrix we can simply add the appropriate row to the appropriate row as the deletion of the used edge is accounted for by the fact that \(1\) has been removed from the corresponding entry of the \((e)t_{D_1}\) row.
\end{proof}
\begin{example}
Here are some examples of the change of relator matrix as described in Lemma \ref{elmat2}.
\[\left[\begin{smallmatrix}
    a & b\\
    c & d
\end{smallmatrix}\right] \rightarrow \left[\begin{smallmatrix}
    a & b\\
    c + a& d + b
\end{smallmatrix}\right]\quad\quad \quad \left[\begin{smallmatrix}
    a & b & c\\
    d & e & f\\
    g & h & i
\end{smallmatrix}\right] \rightarrow \left[\begin{smallmatrix}
    a & b & c\\
    d + g& e + h & f + i\\
    g & h & i
\end{smallmatrix}\right]\]
\end{example}
\begin{lemma}\label{elmat3}
If \(D_1\) is an UDAF digraph obtained from another UDAF digraph \(D_2\), via an elementary UDAF folding of the third type via an edge \(e\in E_{D_1}\), then the matrix obtained from the canonical relator matrix of \(D_1\) by adding the \((e)t_{D_1}\) column to the \((e)s_{D_1}\) column is the canonical relator matrix for \(D_2\).
\end{lemma}
\begin{proof}
This is symmetric with Lemma~\ref{elmat2}.
\end{proof}
\begin{example}
Here are some examples of the change of relator matrix as described in Lemma \ref{elmat3}.
\[\left[\begin{smallmatrix}
    a & b\\
    c & d
\end{smallmatrix}\right] \rightarrow \left[\begin{smallmatrix}
    a & b + a\\
    c & d + c
\end{smallmatrix}\right]\quad\quad \quad \left[\begin{smallmatrix}
    a & b & c\\
    d & e & f\\
    g & h & i
\end{smallmatrix}\right] \rightarrow \left[\begin{smallmatrix}
    a & b + c& c\\
    d & e + f& f\\
    g & h + i & i
\end{smallmatrix}\right]\]
\end{example}

\begin{defn}[UDAF relator matrices]
We say a matrix \(M:S^2\to \Z\) is an \textit{UDAF relator matrix} if it is a relator matrix for an UDAF digraph.
Moreover, we say that two UDAF relator matrices \(A:\{0, 1, \ldots, n-1\}^2\to \N, B:\{0, 1, \ldots, m-1\}^2\to \N\) are \textit{UDAF relator equivalent}, if they are relator matrices for UDAF digraphs which are isomorphic in UDAFG.
\end{defn}
\begin{theorem}[UDAF equivalence of matrices]\label{matequivth}
Let \(A\) and \(B\) be UDAF relator matrices. The matrices \(A\) and \(B\) are UDAF relator equivalent if and only if there is a sequence
\[A=M_0, M_1, \ldots, M_{n}= B,\]
of UDAF relator matrices such that each \(M_i\) can be obtained from \(M_{i-1}\) by doing one of the following:
\begin{enumerate}
    \item Adding a cross centered on the diagonal with \(-1\) on the diagonal and zeros elsewhere.
    \item Adding row \(j\) to row \(i\) when \(i\neq j\) and the \((i, j)\) entry is non-zero.
    \item Adding column \(i\) to column \(j\) when \(i\neq j\) and the \((i, j)\) entry is non-zero. 
    \item The inverse of any of the above operations.
    \item Reordering/relabeling the vertices of the digraph. In other words, suppose that \(M_i:V_i^2\to \Z\) and \(M_{i+1}:V_{i+1}^2\to \Z\) are matrices representing digraphs with vertex sets \(V_i\) and \(V_{i+1}\) respectively.
    Find a bijection \(b:V_i\to V_{i+1}\) such that for all \(v,w\in V_i\) we have \((v, w)M_i= ((v)b, (w)b)M_{i+1}\).
\end{enumerate}
\end{theorem}
\begin{proof}
This follows from Lemma~\ref{elmat1}, Lemma~\ref{elmat2}, and Lemma~\ref{elmat3} together with Corollary~\ref{big groupoid}.
\end{proof}
\begin{corollary}\label{UDAF det}
If \(A\) is an \(n\times n\) UDAF relator matrix and \(B\) is an \(m\times m\) UDAF relator equivalent UDAF relator matrix, then \(\det(A) = (-1)^{n-m}\det(B)\).
\end{corollary}
\begin{proof}
From Theorem~\ref{matequivth}, \(A\) can be obtained from \(B\) by a sequence of processes described there. 
As the determinant of a matrix is negated by operation \(1\) and preserved by operations \(2, 3, 5\), the result follows.
\end{proof}

\begin{example}
There are strongly connected UDAF digraphs which are weak UDAF equivalent but not strong UDAF equivalent. 
From Corollary~\ref{UDAF det} the digraphs with the below relator matrices are not strong UDAF equivalent, but they are weak UDAF equivalent as their UDAF dimension semigroups are both trivial (as the rows generate \(\Z^2\), see Theorem~\ref{UDAF dimension group}).
\[
\left[\begin{smallmatrix}
    0 & 1 \\
    1 & 0 
\end{smallmatrix}\right]\text{ and }
\left[\begin{smallmatrix}
    2 & 1 \\
    1 & 1
\end{smallmatrix}\right].\]
\end{example}
\begin{lemma}\label{safemats}
Suppose that \(M\) is a square matrix with at least 2 rows and non-negative integer entries.
If \(M\) is a relator matrix of a strongly connected digraph, then \(M\) is an UDAF relator matrix (corresponding to a strongly connected UDAF digraph).
\end{lemma}
\begin{proof}
As the entries of \(M\) are non-negative, a digraph with relator matrix \(M\) has at least one loop at each vertex, moreover as the digraph is strongly connected each vertex must also belong to another cycle. 
\end{proof}
\begin{lemma}\label{safematops}
If \(A\) is a relator matrix for a strongly connected UDAF digraph, with no negative entries, then adding a row of \(A\) to another row of \(A\) will result in an UDAF relator equivalent UDAF relator matrix. 
Moreover, the analogous statement using columns instead of rows is also true. 
\end{lemma}
\begin{proof}
Assume without loss of generality that \(A\) is the canonical relator matrix for a strongly connected UDAF digraph \(D\). 
Suppose that \(v,w\in V_D\) and we want to add row \(v\) to row \(w\). Let \(p:\Z_{[0, n]}:\to D\) be an injective homomorphism with \((0)p_V= v\) and \((n)p_V= w\). Then we do the following sequence of operations as allowed by Theorem~\ref{matequivth}:

\begin{enumerate}
    \item Add row \((0)p_V\) to row \((1)p_V\), then add row \((1)p_V\) to row \((2)p_V\), \ldots, then add row \((n-1)p_V\) to row \((n)p_V\).
    \item Subtract row \((n-2)p_V\) from row \((n-1)p_V\), \ldots, then subtract row \((0)p_V\) from row \((1)p_V\).
\end{enumerate}

Let \(B\) be the resulting matrix. The matrix \(B\) is the same as \(A\) except with all of the rows \((0)p_V, (1)p_V,  \ldots, (n-1)p_V\) added to row \((n)p_V\).

Let \(C\) be the matrix obtained from \(A\) by adding row \((0)p_V\) to row \((n)p_V\). Perform the following operations to \(C\):
\begin{enumerate}
    \item Add row \((1)p_V\) to row \((2)p_V\), then add row \((2)p_V\) to row \((3)p_V\), \ldots, then add row \((n-1)p_V\) to row \((n)p_V\).
    \item Subtract row \((n-2)p_V\) from row \((n-1)p_V\), \ldots, then subtract row \((1)p_V\) from row \((2)p_V\).
\end{enumerate}
The matrix obtained is the same as \(C\) except with all of the rows \( (1)p_V,  \ldots, (n-1)p_V\) to row \((n)p_V\), so the matrix obtained is \(B\). 

As long as all the intermediate matrices are UDAF relator matrices, all row operations performed preserve the UDAF relator equivalence type of the matrices by Theorem~\ref{matequivth}.
This is true by Lemma~\ref{safemats}. Thus \(A\) and \(C\) are both equivalent to \(B\), and are hence equivalent to each other as required.
\end{proof}
\begin{lemma}\label{delting crosses}
Suppose that \(A\) is an UDAF relator matrix with a row/column of all 0's except for its diagonal entry which is \(-1\). 
The matrix obtained from this matrix by deleting the row/column together with the column/row containing the \(-1\) diagonal entry is UDAF relator equivalent to \(A\).
\end{lemma}
\begin{proof}
This happens precisely when the matrix corresponds to a digraph possessing a vertex which either has no outgoing edges or no incoming edges. 
These vertices are involved in no bi-infinite walks so can be safely deleted with out changing strong UDAF equivalence type.
\end{proof}

\section{Comparing equivalences}
We have introduced two types of equivalence for UDAF digraphs: weak UDAF equivalence and strong UDAF equivalence. 
We have also discussed strong shift equivalence which is stronger than strong UDAF equivalence (Corollary \ref{SSE implies UDAF cor}). 

In this section we make some observations about shift equivalence. 
Once these are considered we obtain the following diagram showing the relationship between the types of equivalence discussed.

\vspace*{30pt}\begin{rotate}{00}\hspace*{-27pt}\(\text{Weak UDAF Equivalence}\)

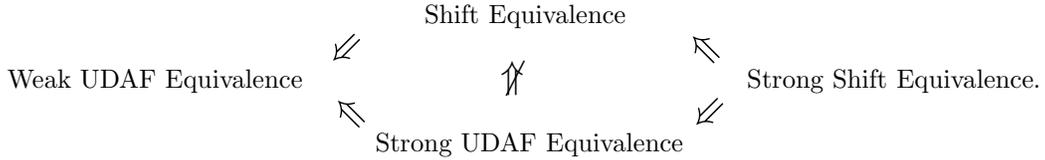
\begin{rotate}{-45}\(\quad\mathlarger{\mathlarger{\mathlarger{\Leftarrow}}}\quad
\)\begin{rotate}{45}\(\text{Strong UDAF Equivalence}
\)\begin{rotate}{45}\(\quad\mathlarger{\mathlarger{\mathlarger{\Leftarrow}}}\quad
\)\begin{rotate}{-45}\(\text{Strong Shift Equivalence}.
\)\end{rotate}\end{rotate}\end{rotate}\end{rotate}

 \begin{rotate}{45}\(\quad\mathlarger{\mathlarger{\mathlarger{\Leftarrow}}}\quad
\)\begin{rotate}{-45}\hspace*{12pt} \(\text{Shift Equivalence}
\)\hspace*{12pt} \begin{rotate}{-45}\(\quad\mathlarger{\mathlarger{\mathlarger{\Leftarrow}}}\quad
\)\end{rotate}\end{rotate}\end{rotate}
\hspace*{65pt}\(\mathlarger{\mathlarger{\mathlarger{\not\Uparrow}}}\)
\end{rotate}
\vspace*{30pt}

Notably the answer to the following question is not given in the above diagram.
\begin{question}
Does shift equivalence imply strong UDAF equivalence?
\end{question}

If this implication does not hold then strong UDAF equivalence can serve as a tool to help identify digraphs which are strong shift equivalent but not shift equivalent. 
However, we unfortunately do not currently know the answer to this question.

We now proceed with the shift equivalence discussion. 
We shall define shift equivalence via dimension modules, for background on this viewpoint see section 7.5 of \cite{lind2021introduction}.
\begin{defn}\label{dimension module defn}
Let \(D\) be a finite digraph. For the purposes of this definition, we view elements of \(\mathbb{Q}^{V_{D}}\) as row vectors.
The subgroup
\[\mathcal{G}_D= \makeset{v\in \mathbb{Q}^{V_{D}}}{\(v \adj_{D}^i\in \Z^{V_{D}}\) for some \(i\in \N\) }\cap \bigcap_{i\in \N} (\mathbb{Q}^{V_{D}})\adj_{D}^i\]
of \(\mathbb{Q}^{V_{D}}\) is call the \textit{dimension group} of \(D\). 
Moreover we define 
\[\mathcal{G}_D^+= \makeset{v\in \mathbb{Q}^{V_{D}}}{\(v \adj_{D}^i\in \N^{V_{D}}\) for some \(i\in \N\) }\cap \bigcap_{i\in \N} (\mathbb{Q}^{V_{D}})\adj_{D}^i.\]
The \textit{dimension module} of \(D\) is defined to be the triple \((\mathcal{G}_D, \mathcal{G}_D^+, \hat{\sigma}_D)\)
where \(\hat{\sigma}_D: \mathcal{G}_D\to \mathcal{G}_D\) is defined by \((v)\hat{\sigma}_D = v \adj_{D}\).
Two finite digraphs are said to be \textit{shift equivalent} if their dimension modules are isomorphic.
\end{defn}
\begin{theorem}
If \(D_1, D_2\) are shift equivalent UDAF digraphs, then they are weak UDAF equivalent.
\end{theorem}
\begin{proof}
It suffices to show that if \((\mathcal{G}_{D_1}, \mathcal{G}_{D_1}^+, \hat{\sigma}_{D_1})\cong (\mathcal{G}_{D_2}, \mathcal{G}_{D_2}^+, \hat{\sigma}_{D_2})\), then \[(\rt(D_1),D_1)\operatorname{Dim}\cong (\rt(D_2),D_2)\operatorname{Dim}.\]

We show in fact, that if \(D\) is an UDAF digraph then
\((\rt(D),D)\operatorname{Dim} \cong \mathcal{G}_{D}^+/ \sim_{\sigma_D}\)
where \(\sigma_D = \hat{\sigma}_{D}\restriction_{\mathcal{G}_{D}^+}\) and \(\sim_{\sigma_D}\) is the congruence generated by \(\sigma_D\). 

For each \(m\in \mathbb{Q}^{V_{D^\circ}}\), let \(m'\) denote the element of \(\mathbb{Q}^{V_D}\) which agrees with \(m\) on elements of \(V_{D}^\circ\) and is \(0\) elsewhere.
For each \(m\in \mathbb{Q}^{V_D}\), let \(m^\circ\) denote the restriction of \(m\) to \(V_{D^\circ}\) and let \(m^-\) denote \(m-(m^{\circ})'\).

\underline{Claim 1:} The map \(\phi: \N^{V_{D^\circ}}/\sim_D\to \mathcal{G}_{D}^+/ \sim_{\sigma_D}\) given by
\[[m]_{\sim_D}\phi= \left[(m')\adj_{D}^{|V_{D}|}\right]_{\sim_{\sigma_D}}\]
is a well-defined monoid homomorphism.\\
\underline{Proof of Claim 1:}
First note that \((\mathbb{Q}^{V_D})\adj_D^{|V_{D}|}= \bigcap_{i\in \N} (\mathbb{Q}^{V_{D}})\adj_{D}^i\) (the dimension of \((\mathbb{Q}^{V_{D}})\adj_{D}^i\) can decrease at most \(|V_{D}|\) times). 
It therefore suffices to show that if \(m_1\sim_D m_2\), then \((m_1')\adj_D^{|V_{D}|} \sim_{\sigma_D} (m_2')\adj_D^{|V_{D}|}\).
By Theorem~\ref{refiner rows}, the congruence \(\sim_D\) is
generated by the pairs
\((\chi_{D, v}, \chi_{D, v}\adj_{D^\circ})\) for \(v\in V_{D^\circ}\). 
So we need only show that if \(v\in V_{D^\circ}\), then 
\[(\chi_{D, v}')\adj_D^{|V_{D}|} \sim_{\sigma_D} (\chi_{D, v}\adj_{D^\circ})'\adj_D^{|V_{D}|}.\]
Note that for \(w\in V_{D^\circ}\), the \(w\)-entry of both \((\chi_{D, v}\adj_{D^\circ})'\) and \((\chi_{D, v})'\adj_{D}\) are the number of edges in \(D\) from \(v\) to \(w\). In particular \((\chi_{D, v}\adj_{D^\circ})' = (((\chi_{D, v})'\adj_{D})^\circ) '\).
For \(w\in V_D\backslash V_{D^\circ}\), the \(w\)-entry of \((\chi_{D, v})'\adj_{D}\) is the number of edges in \(D\) from \(v\) to \(w\), which can only be non-zero if \(w\) is the end of an infinite backwards walk.
In particular the only non-zero entries of \(((\chi_{D, v})'\adj_{D})^-\) are at vertices which are involved in no infinite forwards walks. 
If a vertex is the start of a walk of length \(|V_D|\), then it must be involved in an infinite forwards walk, so we have \(((\chi_{D, v})'\adj_{D})^-\adj_D^{|V_D|}= 0\). 
Hence
\begin{align*}
    (\chi_{D, v}')\adj_{D}^{|V_{D}|}
    &\sim_{\sigma_D} (\chi_{D, v}')\adj_{D}^{|V_{D}|}\sigma_D= (\chi_{D, v}')\adj_{D}^{|V_{D}|+1}\\
    &=( (((\chi_{D, v})'\adj_{D})^\circ) ' + ((\chi_{D, v})'\adj_{D})^-)\adj_{D}^{|V_{D}|}\\
    &=((\chi_{D, v}\adj_{D^\circ})' + ((\chi_{D, v})'\adj_{D})^-)\adj_{D}^{|V_{D}|}\\
    &=(\chi_{D, v}\adj_{D^\circ})'\adj_{D}^{|V_{D}|} + (((\chi_{D, v})'\adj_{D})^-)\adj_{D}^{|V_{D}|}\\
    &=(\chi_{D, v}\adj_{D^\circ})'\adj_{D}^{|V_{D}|}.\diamondsuit
\end{align*}
We will show that \(\phi\) is an isomorphism.

\underline{Claim 2:}
The map \(\phi\) is surjective and moreover for all \(m\in \mathcal{G}_D^+\), we have \((m^\circ)'\adj_D^{|V_D|} = m\adj_D^{|V_D|}\).\\
\underline{Proof of Claim 2:}
 Let \(m\in \mathcal{G}_D^+\) be arbitrary.
Note that if \(v\in V_D\backslash V_{D^\circ}\) and the \(v\)-entry of \(m^-\) is non-zero, then \(v\) is at the end of an infinite backwards walk in \(D\) (as \(m\in \mathcal{G}_D^+\)). 
In particular no vertex with non-zero entry in \(m^-\) is the start of an infinite forwards walk. Thus \((m^-)\adj_{D}^{|V_D|} = 0\) and hence
\[m\sim_{\sigma_D} (m)\adj_{D}^{|V_D|}=((m^\circ)'+m^-)\adj_{D}^{|V_D|}=((m^\circ)')\adj_{D}^{|V_D|}.\]
Thus if \(m_2\in\N^{V_D}\) is arbitrary with \(m_2\sim_{\sigma_D} m\), then we have \([m]_{\sim_{\sigma_D}}=[m_2]_{\sim_{\sigma_D}} = 
([m_2^\circ]_{\sim_D})\phi\) 
is in the image of \(\phi\) as required.\(\diamondsuit\)

It remains to show the injectivity of \(\phi\).
Let \(m_1, m_2\in \N^{V_D^\circ}\) be arbitrary, and suppose that \((m_1')\adj_D^{|V_D|} \sim_{\sigma_D} (m_2')\adj_D^{|V_D|} \). We show that \(m_1\sim_D m_2\). 
By the definition of \(\sim_{\sigma_D}\), we can find \(n\in \N\) and \(a_0, b_0, c_0, d_0, a_1, b_1, c_1, d_1, \ldots, a_n, b_n, c_n, d_n\in \mathcal{G}_D^+\) such that
\begin{enumerate}
    \item \(a_0 + b_0=(m_1')\adj_D^{|V_D|}, a_n + b_n= (m_2')\adj_D^{|V_D|}\).
    \item for all \(i< n\) there is \(k_{i}, k_{i, 2}> |V_D|\) with \((c_i)\sigma_D^{k_{i, 1}} = (a_{i+1})\sigma_D^{k_{i, 2}}\).
    \item for all \(i\leq n\) we have \(a_i + b_i=c_i + d_i\).
    \item for all \(i< n\) we have \(d_i = b_{i+1}\).
\end{enumerate}
Let \(m\in \N\) be greater that \(|V_D|\) and such that for all \(i\leq n\) we have \(a_i\sigma_D^m, b_i\sigma_D^m, c_i\sigma_D^m, d_i\sigma_D^m \in \N^{V_{D}}\).
Throughout the rest of the proof, we will be using the fact that if \(M\in \N^{V_{D^\circ}}\), then \(((M')\adj_D)^\circ=(M)\adj_{D^\circ}\) (as any vertex used in a walk between elements of \(V_{D^\circ}\) must be a vertex of \(V_{D^\circ}\)).

\underline{Claim 3:} For all \(i< n\), \(((a_i+ b_i) \sigma_D^m)^\circ \sim_D ((a_{i+1}+ b_{i+1}) \sigma_D^m)^\circ\).\\
\underline{Proof of Claim 3:}  Let \(i<n\). We obtain
\begin{align*}
    ((a_i+ b_i) \sigma_D^m)^\circ &=  ((c_i+ d_i) \sigma_D^m)^\circ=  ((c_i)\sigma_D^m)^\circ + ((d_i)\sigma_D^m)^\circ \\
    &\sim_D  ((c_i)\sigma_D^m)^\circ\adj_{D^\circ}^{k_{i, 1}} + ((d_i)\sigma_D^m)^\circ\\
 &=  ((((c_i)\sigma_D^m)^\circ)'\adj_{D}^{k_{i, 1}})^\circ + ((d_i)\sigma_D^m)^\circ\\
 &=  ((c_i)\sigma_D^m\adj_{D}^{k_{i, 1}})^\circ + ((d_i)\sigma_D^m)^\circ \quad\text{by claim 2}\\
     &=  ((c_i)\sigma_D^{m+ k_{i, 1}}))^\circ + ((d_i)\sigma_D^m)^\circ=  ((a_{i+1})\sigma_D^{m+ k_{i, 2}})^\circ + ((d_i)\sigma_D^m)^\circ\\
  &=  ((a_{i+1})\sigma_D^{m+ k_{i, 2}})^\circ + ((b_{i+1})\sigma_D^m)^\circ=  ((a_{i+1})\sigma_D^{m}\adj_D^{k_{i, 2}})^\circ + ((b_{i+1})\sigma_D^m)^\circ\\
   &=  ((((a_{i+1})\sigma_D^m)^\circ)'\adj_{D}^{k_{i, 2}})^\circ + ((b_{i+1})\sigma_D^m)^\circ\quad\text{by claim 2}\\
    &=  ((a_{i+1})\sigma_D^m)^\circ\adj_{D^\circ}^{k_{i, 2}} + ((b_{i+1})\sigma_D^m)^\circ\\
    &\sim_{D}  ((a_{i+1}) \sigma_D^m + (b_{i+1})\sigma_D^m)^\circ=  ((a_{i+1} + b_{i+1})\sigma_D^m)^\circ.\diamondsuit
\end{align*}
Thus \(((a_0 + b_0)\adj_D^m)^\circ \sim_D ((a_n + b_n)\adj_D^m)^\circ\).
We also have
\begin{align*}
    m_1 &\sim_D m_1\adj_{D^\circ}^{|V_D|}\adj_{D^\circ}^m=((m_1')\adj_{D}^{|V_D|})^\circ\adj_{D^\circ}^m= (a_0 + b_0)^\circ\adj_{D^\circ}^m\\
    &= (a_0 + b_0)^\circ\adj_{D^\circ}^m= (((a_0 + b_0)^\circ)'\adj_{D}^m)^\circ= ((a_0 + b_0)\sigma_{D}^m)^\circ
\end{align*}
and similarly \(m_2 \sim_D ((a_n + b_n)\sigma_{D}^m)^\circ\).
Hence \(m_1\sim_D m_2\) as required.
\end{proof}

\begin{remark}
The below digraphs (the two-leaf rose and the golden mean shift) are strong UDAF equivalent but not shift equivalent.
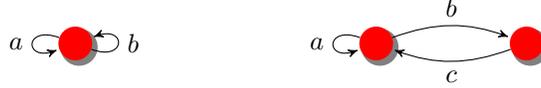
\begin{center}
\begin{tikzpicture}[->,>=stealth',shorten >=1pt,auto,node distance=5cm,on grid, every state/.style={fill=red,draw=none,circular drop shadow,text=white}]

    \node [state, scale = 0.5] (A) at (-3,1) {};

  \node [state, scale = 0.5] (A2) at (1,1) {} ;
  \node [state, scale = 0.5] (B2) at (3,1) {};

 \path [->] (A) edge  [out=160,in=200, loop]
 node  [swap]{\(a\)} (A)
 
 (A) edge  [out=340,in=20, loop]
 node  [swap]{\(b\)} (A)
 
 (A2) edge  [out=160,in=200, loop]
 node  [swap]{\(a\)} (A2)
 (A2) edge [out=20,in=160]
 node {\(b\)} (B2)
 (B2) edge  [out=200,in=340]
 node {\(c\)} (A2);
 \end{tikzpicture}
 \end{center}

\end{remark}
\begin{proof}
It is routine to verify that there is an UDAF folding from the left digraph above to the right digraph above defined by
\[a\to a \quad \text{ and }b\to bc.\] 
Hence these digraphs are strong UDAF equivalent. Alternatively one could show this using Theorem~\ref{matequivth}.

However these digraphs are not shift equivalent as for example the left digraph satisfies \(\mathcal{G}_D\leq \mathbb{Q}^1\) and the right digraph satisfies \(\Z^2 \leq \mathcal{G}_D\) (but \(\Z^2\) doesn't embed in \(\mathbb{Q}\)).
\end{proof}

\part{Open Problems}
As mentioned in the introduction, we are still left without an answer to the following question.
\begin{question}
If two UDAF digraphs are shift equivalent then are they necessarily strong UDAF equivalent?
\end{question}
If not then strong UDAF equivalence can be used to separate shift equivalence and strong shift equivalence in these cases.

We provide a ``small" means of generating UDAFG in the form of elementary UDAF equivalences.
Given Theorem~\ref{autgnr theorem}, the following question is natural.
\begin{question}
Can the category UDAFG be used to build ``small" generating sets for the groups \(\Out(G_{n, n-1})\) (\(n\geq 2\))?
\end{question}
Note that if \(n\geq 2\) is fixed and for all digraphs \(D\) which are strong UDAF equivalent to \(R_n\), we can make a single choice of UDAFG isomorphism \(\phi_D:(\rt(D), D) \to (\rt(R_n), R_n)\), then the set
\[\makeset{\phi_{D_1} \circ e \circ \phi_{D_2}^{-1}}{\(e:(\rt(D_1), D_1) \to (\rt(D_2), D_2)\) is\\ induced by an elementary UDAF folding}\]
is a generating set for \(\Aut_{UDAFG}(R_n) \cong \Out(G_{n, n-1})\). 
One can choose such isomorphisms \(\phi_D\) via the axiom of choice but this is rather unsatisfying.

The problem of checking strong UDAF equivalence via Theorem~\ref{matequivth}, can theoretically be made rather complex due to the need to resize the matrices involved. 
However when carrying out manual examples we have never had a problem with this.
In particular when modifying the matrices to equate Ashley's example with the 2-leaf rose we never need to make a matrix bigger and then smaller again (or vice versa).
Hence the following question.
\begin{question}
Suppose that \(M_1\) and \(M_2\) are UDAF relator equivalent UDAF relator matrices with the same size.
It is necessarily possible to convert \(M_1\) into \(M_2\) as in Theorem~\ref{matequivth} without using operation (1)?
\end{question}
Note that this is trivially satisfied when \(M_1\) and \(M_2\) have size \(1\).

\part{Appendix}
\section{\(\Out(G_{n, n-1})\) acts faithfully on routes}
As mentioned in Remark~\ref{appendix reference}, we now show that the action of the group \(\mathcal{O}_n\) in the set \(\rt(R_n)\) is faithful.

As \(\mathcal{O}_n\) is a group, we need only show that any element which acts as the identity map on \(\rt(R_n)\) is the identity. Unfortunately our proof is highly technical.

\begin{proposition}\label{injectivity}
Let \(n\geq 2\) be an integer. Suppose that \(T= (X_n, Q_T, \pi_T, \lambda_T)\in \mathcal{O}_n\) has synchronizing length \(k_T\).
If the map \(T\) induces on \(\rt(R_n)\) is the identity map, then \(T\) is the identity of \(\mathcal{O}_n\).
\end{proposition}
\begin{proof}
Let \(d_1\geq 2\) be such that for all \(w\in X_n^{d_1}\) and \(q\in Q_T\) we have \(\lambda_T(w, q) \neq \varepsilon\).
Let \(d_2\in \N\) be such that \(|X_n|^{d_2}> d_2 + k_T + d_1\), and let \(d:= d_1+d_2\).

For all non-empty words \(w\in X_n^*\) let \(q_w\) be the unique state of \(T\) for which \(\pi_T(w, q_w)= q_w\). 
If \(w\in X_n^*\) and \(z\in \Z\), then we define \(w_z\in X_n\) by the equalities:
\[w=w_0w_1\ldots w_{|w|-1} \quad \text{ and }\quad w_z=w_{z\text{ mod }|w|}.\]
We also define a map \(\sigma_n: X_n^* \to X_n^*\) by 
\((w)\sigma_n=w_1\ldots w_{|w|-1}w_0\).

\underline{Claim 1:}
For all words \(w\in  X_n^*\), there is a minimum \(s_w\in \{0, 1,2, \ldots |w| - 1\}\) such that \(\lambda_T(w, q_{w})=(w)\sigma_n^{s_w}\).\\
\underline{Proof of Claim 1:}
Let \(w\in X_n^*\) and assume without loss of generality that \(w\) is not a power of any other word. 
Let \(b\in X_n\) be such that \(w\neq b\). 
As \(T\) fixes all routes, it follows that it fixes the route which consists of the word \(w\) repeated infinitely many times.
In particular, there is \(m\geq 1\) and \(s_w\in \Z\) such that \(\lambda_T(w, q_{w})=(w^m)\sigma_n^{s_w}\).
It therefore suffices to show that this \(m\) is \(1\).

Suppose for a contradiction that \(m\geq 2\). 
Consider the word \(B:= b^{|w|}w^{k_T}w^{2k_T +2}\).
Note that \(w^{3k_T+4}\) is not a contiguous substring of \(B\).
However, as \(m\geq 2\), it follows that if \(B\) is read from any state, a word with \(w^{ 4k_T +4}\) as a contiguous substring is written.
Thus \(T\) does not send the route consisting of \(B\) infinitely many times to itself, a contradiction.\(\diamondsuit\)

\underline{Claim 2:} If \(w\in X_n^{{k_T}+d}\), then \(s_w=0\).\\
\underline{Proof of Claim 2:} Suppose for a contradiction that \(s_w>0\). As \(T\) has no incomplete response, there exists a word \(v \in X_n^{d_1}\) such that 
\[w_{s_w}':=(\lambda_T(v,q_{w}))_0\] 
is not equal to \(w_{s_w}\). Let \(x\in X_n^{d_2}\) be such that \(x\) is not a contiguous subword of \(ww\) (this is possible by the choice of \(d_2\)).
Let \(w^{-}\in X_n^{k_T}\) be a suffix of \(w\) and let 
 \(b := v xw^{-}w^2\).
We have that \(q_b=q_{w}\) and \(|b|=3(d+{k_T})\).
Note that \(\lambda_T(b,q_{w}) = (b)\sigma_n^{s_b}\). 
Now choose some \(w_{s_w - 1}'\in X_n\backslash \{w_{s_w - 1}\}\).
Let
 \[B:=wb{w_{s_{w-1}}'}^{|b|}w\]
and note that the following things are true:
\begin{enumerate}
\item \(|B| = (d+{k_T})+3(d+{k_T})+3(d+{k_T})+(d+{k_T})=8(d+{k_T})\).
\item \(q_B=q_w\).
\item \(\lambda_T(B,q_{w}) = (B)\sigma_n^{s_B}\).
\item For \(i\in \{0,1,2 \ldots 8(d+{k_T})-1\}\) we have
\[B_i= \left\{\begin{array}{cc}
w_i    &  0\leq i < d+{k_T}\\
b_{i-(d+{k_T})}     & d+{k_T}\leq i < 4(d+{k_T})\\
w_{s_w-1}' & 4(d+{k_T}) \leq i<7(d+{k_T})\\
w_{i-7(d+{k_T})}& 7(d+{k_T}) \leq i< 8(d+{k_T})
\end{array}\right\}\]
    \item For \(i\in \{0,1, 2 \ldots 8(d+{k_T})-1\}\) we have
\[(\lambda_T(B,q_{w}))_i= \left\{\begin{array}{cc}
(\lambda_T(w,q_{w}))_i    &  0\leq i < d+{k_T}\\
(\lambda_T(b,q_{w}))_{i-(d+{k_T})}     & d+{k_T}\leq i < 4(d+{k_T})\\
(\lambda_T({w_{s_w-1}'}^{|b|}w,q_{w}))_{i-4(d+{k_T})}& 4(d+{k_T}) \leq i< 8(d+{k_T})
\end{array}\right\}\]
\end{enumerate}
We have \(\lambda_T(B, q_{w}) = (B)\sigma_n^{s_B}\), we now show that all possible values of \(s_B\) lead to contradictions:
\begin{enumerate}
\item If \(s_B= 0\), then \(\lambda_T(w, q_{w}) = w\) and thus \(s_w=0\), a contradiction as we assumed that \(s_w>0\).
\item If \(0< s_B \leq 3(d+{k_T})\), then by the choice of \(b\) we have
\begin{align*}
    w_{s_w-1} &= (\lambda_T(b, q_{w}))_{3(d + {k_T}) - 1}= (\lambda_T(B, q_{w}))_{4(d + {k_T}) - 1}=(B\sigma_n^{s_B})_{4(d+{k_T})- 1}\\
    &=B_{4(d+{k_T})- 1+s_B} = w_{s_w-1}'
\end{align*}
a contradiction.
\item If \(3(d+{k_T}) < s_B \leq 6(d+{k_T})\), then
\[w_{s_w-1} = (\lambda_T(B,q_{w}))_{d+{k_T}-1}=(B\sigma_n^{s_B})_{d+{k_T}-1} =B_{d+{k_T}-1+s_B}= w_{s_w-1}' \]
 a contradiction.
\item Suppose that \(6(d + {k_T})< s_B <7(d+ {k_T})\). For all \(i\) such that \(0\leq i< d_2\), we have
\[(\lambda_T(B, q_w))_{d+k_T+d_1+i -s_B}=(B\sigma_n^{s_B})_{d+k_T+d_1+i-s_B}=(B)_{d+k_T+d_1+i}=x_i\]
Thus \(\lambda_T(B, q_w)\) has a contiguous substring equal to \(x\) strictly between positions \(2(d+k_T)\) and \(4(d+k_T)\). So \(\lambda_T(b, q_w)\) has a contiguous substring equal to \(x\) strictly between positions \((d+k_T)\) and \(3(d+k_T)\). As between these positions, the string \(\lambda_T(b, q_w)\) is a rotation of the word \(w^2\), this contradicts the definition of \(x\).
\item Finally suppose \(7(d+k)\leq s_B < 8(d+{k_T})\). 
Using the fact that \(B\) both begins and ends with the word \(w\) we have:
\begin{align*}
w_{s_w}' &= (\lambda_T(B, q_{w}))_{d+{k_T}}\\
&=(B\sigma_n^{s_B})_{d+{k_T}}\\
&= B_{d+{k_T}+s_B}\\
&= B_{s_B}\\
&=(B\sigma_n^{s_B})_0\\
&= (\lambda_T(B, q_{w}))_0\\
&= w_{s_w}
\end{align*}
\end{enumerate}
The claim has now been proved \(\diamondsuit\)\\

If \(S\) is a non-empty collection of elements of \(X_n^* \cup X_n^\omega\), then we denote the longest common prefix of these words by \(\mathcal{P}(S)\).

We now show that for any state \(q\in Q\) and letter \(x\in X_n\) that \(\lambda_T(x, q)=x\).
As \(T\) has no incomplete response it suffices to show that \(x=\mathcal{P}(\lambda_T(xX_n^\omega, q))\). Let \(v\in X_n^{k_T}\) be such that \(q_v=q\).\\
\((\leq):\) Let \(\bar{x}\in xX_n^\omega\) and let \(x'\) be the first \(d\) letters of \(\bar{x}\). We have by definition of \(d\) that \(\lambda_T(x', q) \neq \varepsilon\). 
It follows from the claim that \(\lambda_T(x'v, q)=x'v\) so \(x\leq \lambda_T(x', q) < \lambda_T(\bar{x}, q)\). 
We now have that \(x\leq\mathcal{P}(\lambda_T(xX_n^\omega, q))\).\\
\((\geq):\)  Let \(a, b\in X_n^d\) be such that \(a_0 = b_0 =x, a_1 \neq b_1\). Let \(\bar{x}\in avX_n^\omega\) and \(\bar{y}\in bvX_n^\omega\). 
By the claim we have
\[\lambda_T(av, q)=av,\quad \lambda_T(bv, q)=bv.\]
So \(\mathcal{P}(\bar{x}, \bar{y}) = x\) and the result follows.
\end{proof}

\section{Ashley's example}
In this section we give an example as to how Theorem~\ref{matequivth} can be used to verify the strong UDAF equivalence of matrices.
\begin{example}[Jonathan Ashley's example, see Example 2.27 of \cite{kitchens1997symbolic}]
We show that the 2-leaf rose and Ashley's example

\[\left[\begin{smallmatrix}
    1 & 1& 0& 0& 0& 0& 0& 0\\
    0 & 1& 1& 0& 0& 0& 0& 0\\
    0 & 0& 0& 1& 0& 0& 1& 0\\
    0 & 0& 0& 0& 1& 0& 0& 1\\
    0 & 0& 1& 0& 0& 1& 0& 0\\
    0 & 0& 0& 0& 1& 0& 1& 0\\
    0 & 0& 0& 1& 0& 0& 0& 1\\
    1 & 0& 0& 0& 0& 1& 0& 0
\end{smallmatrix}\right]\]
are strong UDAF equivalent. It is currently unknown if these are strong shift equivalent. 
The above matrix is notably the sum of the two permutation matrices (12345678) and (1)(2)(374865).
\end{example}
\begin{proof}
We will show that both of the corresponding relator matrices are UDAF relator equivalent to the matrix
\(\left[\begin{smallmatrix}
    0 & 1& 0& 0\\
    0 & 0& 1& 0\\
    0 & 0& 0& 1\\
    1 & 0& 0& 0
\end{smallmatrix}\right].\)
 We first transform the 2-leaf rose:
 \allowdisplaybreaks
\begin{align*}
\left[\begin{smallmatrix}
    1
\end{smallmatrix}\right] &\rightarrow
\left[\begin{smallmatrix}
    1 & 0\\
    0 & -1
\end{smallmatrix}\right] \text{ add a cross}\\
&\leftarrow
\left[\begin{smallmatrix}
    1 & 1\\
    0 & -1
\end{smallmatrix}\right] \text{add row 2 to row 1 (as (1,2) entry is 1)}
\\
&\leftarrow
\left[\begin{smallmatrix}
    0 & 1\\
    1 & -1
\end{smallmatrix}\right] \text{add column 2 to column 1 (as (2, 1) entry is 1)}
\\
&\rightarrow
\left[\begin{smallmatrix}
    0 & 1\\
    1 & 0
\end{smallmatrix}\right]\text{ add row 1 to row 2 (as (2, 1) entry was 1)}\\
&\rightarrow
\left[\begin{smallmatrix}
    0 & 1 & 0\\
    1 & 0 & 0\\
    0 & 0 & -1
\end{smallmatrix}\right]\text{ add cross}\\
&\rightarrow
\left[\begin{smallmatrix}
    0 & 1 & 0 & 0\\
    1 & 0 & 0 & 0\\
    0 & 0 & -1& 0\\
    0 & 0 & 0& -1
\end{smallmatrix}\right]\text{ add cross}\\
&\leftarrow
\left[\begin{smallmatrix}
    0 & 1 & 0 & 0\\
    1 & 0 & 1 & 0\\
    0 & 0 & -1& 0\\
    0 & 0 & 0& -1
\end{smallmatrix}\right]\text{ add row 3 to row 2 (as (2, 3) entry is 1)}\\
&\leftarrow
\left[\begin{smallmatrix}
    0 & 1 & 0 & 0\\
    0 & 0 & 1 & 0\\
    1 & 0 & -1& 0\\
    0 & 0 & 0& -1
\end{smallmatrix}\right]\text{add column 3 to column 1 (as (3, 1) entry is 1)}\\
&\rightarrow
\left[\begin{smallmatrix}
    0 & 1 & 1 & 0\\
    0 & 0 & 1 & 0\\
    1 & 0 & -1& 0\\
    0 & 0 & 0& -1
\end{smallmatrix}\right]\text{add row 2 to row 1 (as (1, 2) entry was 1)}\\
&\rightarrow
\left[\begin{smallmatrix}
    0 & 1 & 1 & 0\\
    0 & 0 & 1 & 0\\
    1 & 0 & 0& 0\\
    0 & 0 & 0& -1
\end{smallmatrix}\right]\text{add column 1 to column 3 (as (1, 3) entry was 1)}\\
&\leftarrow
\left[\begin{smallmatrix}
    0 & 1 & 1 & 0\\
    0 & 0 & 1 & 0\\
    1 & 0 & 0& 1\\
    0 & 0 & 0& -1
\end{smallmatrix}\right]\text{add row 4 to row 3 (as (3, 4) entry is 1)}
\\
&\leftarrow
\left[\begin{smallmatrix}
    0 & 1 & 1 & 0\\
    0 & 0 & 1 & 0\\
    0 & 0 & 0& 1\\
    1 & 0 & 0& -1
\end{smallmatrix}\right]\text{add column 4 to column 1 (as (4, 1) entry is 1)}\\
&\rightarrow
\left[\begin{smallmatrix}
    0 & 1 & 1 & 1\\
    0 & 0 & 1 & 0\\
    0 & 0 & 0& 1\\
    1 & 0 & 0& -1
\end{smallmatrix}\right]\text{add row 3 to row 1 (as (1, 3) entry was 1)}\\
&\rightarrow
\left[\begin{smallmatrix}
    0 & 1 & 1 & 1\\
    0 & 0 & 1 & 0\\
    0 & 0 & 0& 1\\
    1 & 0 & 0& 0
\end{smallmatrix}\right]\text{add column 1 to column 4 (as (1, 4) entry was 1)}\\
&\leftarrow
\left[\begin{smallmatrix}
    0 & 1 & 0 & 1\\
    0 & 0 & 1 & 0\\
    0 & 0 & 0& 1\\
    1 & 0 & 0& 0
\end{smallmatrix}\right]\text{add column 2 to column 3 (Lemma~\ref{safematops})}\\
&\leftarrow
\left[\begin{smallmatrix}
    0 & 1 & 0 & 0\\
    0 & 0 & 1 & 0\\
    0 & 0 & 0& 1\\
    1 & 0 & 0& 0
\end{smallmatrix}\right]\text{add column 2 to column 4 (Lemma~\ref{safematops})}.
\end{align*}
We now convert the Ashley example
\allowdisplaybreaks
\begin{align*}
    \left[\begin{smallmatrix}
    0 & 1& 0& 0& 0& 0& 0& 0\\
    0 & 0& 1& 0& 0& 0& 0& 0\\
    0 & 0& -1& 1& 0& 0& 1& 0\\
    0 & 0& 0& -1& 1& 0& 0& 1\\
    0 & 0& 1& 0& -1& 1& 0& 0\\
    0 & 0& 0& 0& 1& -1& 1& 0\\
    0 & 0& 0& 1& 0& 0& -1& 1\\
    1 & 0& 0& 0& 0& 1& 0& -1
\end{smallmatrix}\right]
&\rightarrow
    \left[\begin{smallmatrix}
    0 & 1& 0& 0& 0& 0& 0& 0\\
    0 & 0& 0& 1& 0& 0& 1& 0\\
    0 & 0& -1& 1& 0& 0& 1& 0\\
    0 & 0& 0& -1& 1& 0& 0& 1\\
    0 & 0& 1& 0& -1& 1& 0& 0\\
    0 & 0& 0& 0& 1& -1& 1& 0\\
    0 & 0& 0& 1& 0& 0& -1& 1\\
    1 & 0& 0& 0& 0& 1& 0& -1
\end{smallmatrix}\right] \text{add row 3 to row 2}\\
&\rightarrow
    \left[\begin{smallmatrix}
    0 & 1& 0& 0& 0& 0& 0& 0\\
    0 & 0& 0& 1& 0& 0& 1& 0\\
    0 & 0& -1& 1& 0& 0& 1& 0\\
    0 & 0& 0& -1& 1& 0& 0& 1\\
    0 & 0& 0& 1& -1& 1& 1& 0\\
    0 & 0& 0& 0& 1& -1& 1& 0\\
    0 & 0& 0& 1& 0& 0& -1& 1\\
    1 & 0& 0& 0& 0& 1& 0& -1
\end{smallmatrix}\right] \text{add row 3 to row 5}\\
&\rightarrow
    \left[\begin{smallmatrix}
    0 & 1& 0& 0& 0& 0& 0\\
    0 & 0& 1& 0& 0& 1& 0\\
    0 & 0& -1& 1& 0& 0& 1\\
    0 & 0& 1& -1& 1& 1& 0\\
    0 & 0& 0& 1& -1& 1& 0\\
    0 & 0& 1& 0& 0& -1& 1\\
    1 & 0& 0& 0& 1& 0& -1
\end{smallmatrix}\right] \text{delete column 3 (Lemma~\ref{delting crosses})}\\
&\rightarrow
\left[\begin{smallmatrix}
    0 & 1& 0& 0& 0& 0& 0\\
    0 & 0& 1& 0& 0& 1& 0\\
    0 & 0& 0& 0& 1& 1& 1\\
    0 & 0& 1& -1& 1& 1& 0\\
    0 & 0& 0& 1& -1& 1& 0\\
    0 & 0& 1& 0& 0& -1& 1\\
    1 & 0& 0& 0& 1& 0& -1
\end{smallmatrix}\right]\text{add row 4 to row 3}\\
&\rightarrow
\left[\begin{smallmatrix}
    0 & 1& 0& 0& 0& 0& 0\\
    0 & 0& 1& 0& 0& 1& 0\\
    0 & 0& 0& 0& 1& 1& 1\\
    0 & 0& 1& -1& 1& 1& 0\\
    0 & 0& 1& 0& 0& 2& 0\\
    0 & 0& 1& 0& 0& -1& 1\\
    1 & 0& 0& 0& 1& 0& -1
\end{smallmatrix}\right]\text{add row 4 to row 5}\\
&\rightarrow
\left[\begin{smallmatrix}
    0 & 1& 0&  0& 0& 0\\
    0 & 0& 1&  0& 1& 0\\
    0 & 0& 0&  1& 1& 1\\
    0 & 0& 1&  0& 2& 0\\
    0 & 0& 1&  0& -1& 1\\
    1 & 0& 0&  1& 0& -1
\end{smallmatrix}\right]\text{delete column 4 (Lemma~\ref{delting crosses})}\\
&\rightarrow
\left[\begin{smallmatrix}
    0 & 1& 0&  0& 0& 0\\
    0 & 0& 1&  0& 1& 0\\
    0 & 0& 0&  1& 1& 1\\
    0 & 0& 1&  0& 2& 0\\
    1 & 0& 1&  1& -1& 0\\
    1 & 0& 0&  1& 0& -1
\end{smallmatrix}\right]\text{add row 6 to row 5}\\
&\rightarrow
\left[\begin{smallmatrix}
    0 & 1& 0&  0& 0& 0\\
    0 & 0& 1&  0& 1& 0\\
    1 & 0& 0&  2& 1& 0\\
    0 & 0& 1&  0& 2& 0\\
    1 & 0& 1&  1& -1& 0\\
    1 & 0& 0&  1& 0& -1
\end{smallmatrix}\right]\text{add row 6 to row 3}\\
&\rightarrow
\left[\begin{smallmatrix}
    0 & 1& 0&  0& 0\\
    0 & 0& 1&  0& 1\\
    1 & 0& 0&  2& 1\\
    0 & 0& 1&  0& 2\\
    1 & 0& 1&  1& -1\\
\end{smallmatrix}\right]\text{delete column 6 (Lemma~\ref{delting crosses})}\\
&\rightarrow
\left[\begin{smallmatrix}
    0 & 1& 0&  0& 0\\
    1 & 0& 2&  1& 0\\
    1 & 0& 0&  2& 1\\
    0 & 0& 1&  0& 2\\
    1 & 0& 1&  1& -1\\
\end{smallmatrix}\right]\text{add row 5 to row 2}\\
&\rightarrow
\left[\begin{smallmatrix}
    0 & 1& 0&  0& 0\\
    1 & 0& 2&  1& 0\\
    2 & 0& 1&  3& 0\\
    0 & 0& 1&  0& 2\\
    1 & 0& 1&  1& -1\\
\end{smallmatrix}\right]\text{add row 5 to row 3}\\
&\rightarrow
\left[\begin{smallmatrix}
    0 & 1& 0&  0& 0\\
    1 & 0& 2&  1& 0\\
    2 & 0& 1&  3& 0\\
    2 & 0& 3&  2& 0\\
    1 & 0& 1&  1& -1\\
\end{smallmatrix}\right]\text{add row 5 to row 4 (twice)}\\
&\rightarrow
\left[\begin{smallmatrix}
    0 & 1& 0&  0\\
    1 & 0& 2&  1\\
    2 & 0& 1&  3\\
    2 & 0& 3&  2\\
\end{smallmatrix}\right]\text{delete column 5 (Lemma~\ref{delting crosses})}\\
&\leftarrow
\left[\begin{smallmatrix}
    0 & 1& 0&  0\\
    1 & 0& 2&  0\\
    2 & 0& 1&  1\\
    2 & 0& 3&  0\\
\end{smallmatrix}\right]\text{add column 1 to column 4 (Lemma~\ref{safematops})}\\
&\leftarrow
\left[\begin{smallmatrix}
    0 & 1& 0&  0\\
    1 & 0& 2&  0\\
    0 & 0& 1&  1\\
    2 & 0& 3&  0\\
\end{smallmatrix}\right]\text{add column 4 to column 1 (twice) (Lemma~\ref{safematops})}\\
&\leftarrow
\left[\begin{smallmatrix}
    0 & 1& 0&  0\\
    1 & 0& 2&  0\\
    0 & 0& 0&  1\\
    2 & 0& 3&  0\\
\end{smallmatrix}\right]\text{add column 4 to column 3 (Lemma~\ref{safematops})}\\
&\leftarrow
\left[\begin{smallmatrix}
    0 & 1& 0&  0\\
    1 & 0& 1&  0\\
    0 & 0& 0&  1\\
    2 & 0& 1&  0\\
\end{smallmatrix}\right]\text{add column 1 to column 3 (Lemma~\ref{safematops})}\\
&\leftarrow
\left[\begin{smallmatrix}
    0 & 1& 0&  0\\
    0 & 0& 1&  0\\
    0 & 0& 0&  1\\
    1 & 0& 1&  0\\
\end{smallmatrix}\right]\text{add column 3 to column 1 (Lemma~\ref{safematops})}\\
&\leftarrow
\left[\begin{smallmatrix}
    0 & 1& 0&  0\\
    0 & 0& 1&  0\\
    0 & 0& 0&  1\\
    1 & 0& 0&  0\\
\end{smallmatrix}\right]\text{add row 2 to row 4 (Lemma~\ref{safematops})}\\
\end{align*}
\end{proof}
\bibliographystyle{plainyr}
\bibliography{bib.bib}

\def\cprime{$'$} \def\cprime{$'$} \def\cprime{$'$} \def\cprime{$'$}
\begin{thebibliography}{10}

\bibitem{GNS2000}
Rostislav.~I. Grigorchuk, Volodia.~V. Nekrashevich, and Vitaly.~I.
  Sushchanski{\u\i}.
\newblock Automata, dynamical systems, and groups.
\newblock {\em Proc. Steklov Inst. Math}, 231:128--203, 2000.

\bibitem{Brin2004}
Matthew~G. Brin.
\newblock Higher dimensional {T}hompson groups.
\newblock {\em Geometriae Dedicata}, 108(1):163--192, Oct 2004.

\bibitem{Brin2005}
Matthew~G. Brin.
\newblock Presentations of higher dimensional {T}hompson groups.
\newblock {\em Journal of Algebra}, 284(2):520–558, Feb 2005.

\bibitem{Bleak2010}
Collin Bleak and Daniel Lanoue.
\newblock A family of non-isomorphism results.
\newblock {\em Geometriae Dedicata}, 146(1):21--26, Jun 2010.

\bibitem{pardo2010isomorphism}
Enrique Pardo.
\newblock The isomorphism problem for {H}igman-{T}hompson groups, 2010.

\bibitem{autgnr}
Collin Bleak, Peter Cameron, Yonah Maissel, Andrés Navas, and Feyishayo
  Olukoya.
\newblock The further chameleon groups of {R}ichard {T}hompson and {G}raham
  {H}igman: Automorphisms via dynamics for the {H}igman groups ${G}_{n,r}$,
  2016.

\bibitem{Smithnormalform}
Richard~P. Stanley.
\newblock Smith normal form in combinatorics.
\newblock {\em Journal of Combinatorial Theory, Series A}, 144:476--495, 2016.
\newblock Fifty Years of the Journal of Combinatorial Theory.

\bibitem{thumann2016operad}
Werner Thumann.
\newblock Operad groups and their finiteness properties, 2016.

\bibitem{quick2019}
Martyn Quick.
\newblock Permutation-based presentations for {B}rin's higher-dimensional
  {T}hompson groups $n{V}$, 2019.

\bibitem{belk2020automorphisms}
James Belk, Collin Bleak, Peter~J Cameron, and Feyishayo Olukoya.
\newblock Automorphisms of shift spaces and the {H}igman-{T}hompson groups: the
  two-sided case.
\newblock {\em arXiv preprint arXiv:2006.01466}, 2020.

\bibitem{bleak2020automorphisms}
Collin Bleak, Peter~J Cameron, and Feyishayo Olukoya.
\newblock Automorphisms of shift spaces and the {H}igman--{T}homspon groups:
  the one-sided case.
\newblock {\em arXiv preprint arXiv:2004.08478}, 2020.

\bibitem{autnv}
Luke Elliott.
\newblock A description of \(\operatorname{Aut}(d{V}_n)\) and
  \(\operatorname{Out}(d{V}_n)\) using transducers, 2020.

\bibitem{higman1974finitely}
Graham Higman.
\newblock {\em Finitely presented infinite simple groups}, volume~8.
\newblock Department of Pure Mathematics, Department of Mathematics, IAS,
  Australian, 1974.

\bibitem{kitchens1997symbolic}
Bruce~P Kitchens.
\newblock {\em Symbolic dynamics: one-sided, two-sided and countable state
  Markov shifts}.
\newblock Springer Science \& Business Media, 1997.

\bibitem{lind2021introduction}
Douglas Lind and Brian Marcus.
\newblock {\em An introduction to symbolic dynamics and coding}.
\newblock Cambridge university press, 2021.

\end{thebibliography}

\end{document}